\documentclass[11pt,reqno]{amsart}

\usepackage{amssymb}
\usepackage{latexsym}
\usepackage{amsbsy}
\usepackage{amsfonts}
\usepackage{pstricks}
\usepackage{auto-pst-pdf}
\usepackage{pst-fill,pst-grad}
\usepackage{enumitem}

\newtheorem{theorem}[equation]{Theorem}
\newtheorem{conj}[]{Conjecture}
\newtheorem{proposition}[equation]{Proposition}
\newtheorem{lemma}[equation]{Lemma}

\theoremstyle{definition}
\newtheorem{auxremark}[equation]{Remark}
\newenvironment{remark}{%
  \begin{auxremark}%
  }{%
   \hfill$\diamondsuit$%
    \end{auxremark}
  }

\numberwithin{equation}{section}

\def\AArm{\fam0 \rm}%
\newdimen\AAdi%
\newbox\AAbo%
\def\AAk#1#2{\setbox\AAbo=\hbox{#2}\AAdi=\wd\AAbo\kern#1\AAdi{}}%

\newcommand{\BBone}{{\ensuremath{{\AArm 1\AAk{-.8}{I}I}}}}

\makeatletter
\def\eqlabel#1{\def\@currentlabel{#1}}

\def\formula#1{\def\@tempa{#1}\let\@tempb\theequation\def\theequation{%
\hbox{#1}}\def\@currentlabel{(\theequation)}$$}
\def\endformula{\leqno\hbox{(\@tempa)}$$\@ignoretrue\let\theequation\@tempb}

\def\given{\hskip5\p@\relax\vrule\@width.4\p@\hskip5\p@\relax}

\newcommand{\open}[1]{%
\par\normalfont\topsep6\p@\@plus6\p@\trivlist\item[\hskip\labelsep\itshape#1%
\@addpunct{.}]\ignorespaces}

\DeclareRobustCommand{\close}[1]{%
  \ifmmode 
  \else \leavevmode\unskip\penalty9999 \hbox{}\nobreak\hfill
  \fi
  \quad\hbox{$#1$}}
\makeatother

\newlength{\toskip}

\def\<{\langle}
\def\>{\rangle}

\def \R {{\mathbb R}}

\def \E {{\mathbb E}}

\def \L {{\mathbb L}}

\def \Var {\textrm{Var}}

\newcommand*{\dmu}{\,d\mu}

\begin{document}
\date{\today}

\title[Poincar\'e in a punctured domain.]{Ornstein-Uhlenbeck pinball: I. Poincar\'e inequalities in a punctured domain.}

\author[E. Boissard]{\textbf{\quad {Emmanuel} Boissard $^{\heartsuit}$ \, \, }}
\address{{\bf {Emmanuel} BOISSARD},\\ Weierstrass Institute.\\ 
Mohrenstrasse 39, 10117 Berlin, Germany.}\email{emmanuel.boissard@wias-berlin.de}

 \author[P. Cattiaux]{\textbf{\quad {Patrick} Cattiaux $^{\spadesuit}$ \, \, }}
\address{{\bf {Patrick} CATTIAUX},\\ Institut de Math\'ematiques de Toulouse. CNRS UMR 5219. \\
Universit\'e Paul Sabatier,
\\ 118 route
de Narbonne, F-31062 Toulouse cedex 09.} \email{cattiaux@math.univ-toulouse.fr}

\author[A. Guillin]{\textbf{\quad {Arnaud} Guillin $^{\diamondsuit, \clubsuit}$}}
\address{{\bf {Arnaud} GUILLIN},\\ Laboratoire de Math\'ematiques, CNRS UMR 6620, Universit\'e Blaise Pascal,
avenue des Landais, F-63177 Aubi\`ere.} \email{guillin@math.univ-bpclermont.fr}

 \author[L. Miclo]{\textbf{\quad {Laurent} Miclo $^{\spadesuit}$ \, \, }}
\address{{\bf {Laurent} MICLO},\\ Institut de Math\'ematiques de Toulouse. CNRS UMR 5219. \\
Universit\'e Paul Sabatier,
\\ 118 route
de Narbonne, F-31062 Toulouse cedex 09.} \email{miclo@math.univ-toulouse.fr}

\maketitle

 \begin{center}

\textsc{$^{\heartsuit}$ Weierstrass Institute}
\smallskip

 \textsc{$^{\spadesuit}$  Universit\'e de Toulouse}
\smallskip

\textsc{$^{\diamondsuit}$ Universit\'e Blaise Pascal}
\smallskip

\textsc{$^{\clubsuit}$ Institut Universitaire de France}
\smallskip

 \end{center}

\begin{abstract}
In this paper we study the Poincar\'e constant for the Gaussian measure restricted to $D=\R^d - B(y,r)$ where $B(y,r)$ denotes the Euclidean ball with center $y$ and radius $r$, and $d\geq 2$. We also study the case of the $l^\infty$ ball (the hypercube). This is the first step in the study of the asymptotic behavior of a $d$-dimensional Ornstein-Uhlenbeck process in the presence of obstacles with elastic normal reflections (the Ornstein-Uhlenbeck pinball) we shall study in a companion paper.
\end{abstract}
\bigskip

\bigskip

\bigskip

\section{Introduction.}\label{Intro}

This paper is the first of a series of at least two. We intend to study the asymptotic behavior of a $d$-dimensional Ornstein-Uhlenbeck process in the presence of obstacles with elastic normal reflections (looking like a random pinball). The choice of an Ornstein-Uhlenbeck is made for simplicity as it captures already all the new difficulties of this setting, but a general gradient drift diffusion process (satisfying an ordinary Poincar\'e inequality) could be considered.

All over the paper we assume that $d\geq 2$. We shall mainly consider the case where the obstacles are non overlapping balls of radius $r_i$ and centers $(x_i)_{1\leq
i \leq N \leq +\infty}$, as overlapping obstacles could produce disconnected domains and thus non uniqueness of invariant measures (as well as no Poincar\'e inequality), but we should also look at ``soft obstacles'' as in Sznitman's book
\cite{ASol}. We shall also look at different forms of obstacles when it can enlighten the discussion.

To be more precise, consider for $1\leq N \leq +\infty$, $\mathcal X =(x_i)_{1\leq i \leq N \leq
+\infty}$ a collection of points, and $(r_i)_{1\leq i \leq N \leq
+\infty}$ a collection of non negative real numbers, satisfying
\begin{equation}\label{eqobstac}
|x_i-x_j|> r_i+r_j \textrm{  for } i \neq j \, .
\end{equation}
The Ornstein-Uhlenbeck pinball will be given by 
the following stochastic differential system with reflection
\begin{equation}\label{eqOUR}
\left\{
\begin{array}{l}
dX_t =  dW_t \, - \, \lambda \, X_t \, dt \, + \, \sum_i \, (X_t - x_i) \, dL_t^i  \, ,\\
L_t^i = \int_0^t \, \BBone_{|X_s - x_i|=r_i} \, dL_s^i .
\end{array}
\right.
\end{equation}
Here $W$ is a standard Wiener process and  we assume that $\mathbb P(|X_0-x_i|\geq r_i \textrm{ for all } i)=1$. $L^i$ is
the local time description of the elastic and normal reflection of the process when it hits
$B(x_i,r_i)$.
\begin{center}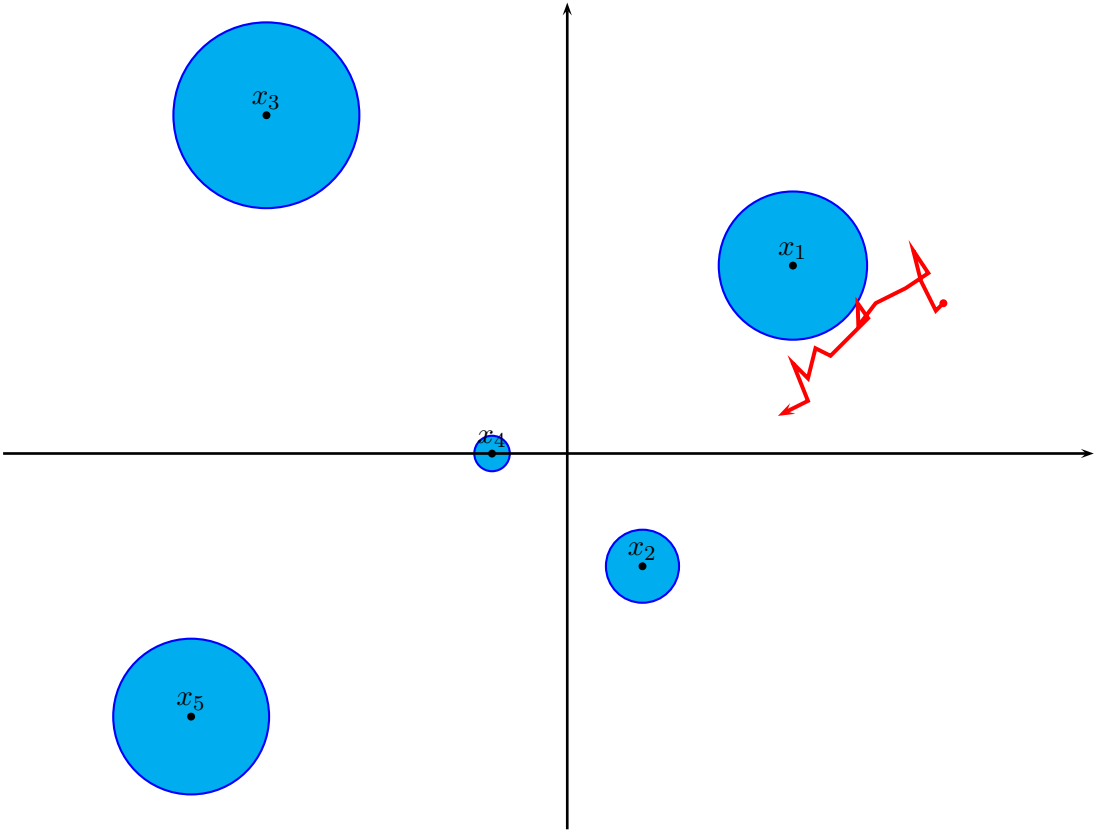
\begin{figure}[!h]
\begin{pspicture}(-7.5,-7.7)(7.5,7.7)
\psset{dotsize=3pt}
\psdots(3,2.5)
\pscircle[linecolor=blue,fillcolor=cyan,fillstyle=solid](3,2.5){1}
\pscircle[linecolor=blue,fillcolor=cyan,fillstyle=solid](1,-1.5){0.5}
\pscircle[linecolor=blue,fillcolor=cyan,fillstyle=solid](-4,4.5){1.25}
\pscircle[linecolor=blue,fillcolor=cyan,fillstyle=solid](-1,0){0.25}
\pscircle[linecolor=blue,fillcolor=cyan,fillstyle=solid](-5,-3.5){1.05}
\psline[linewidth=1pt,linecolor=black]{->}(-7.5,0)(7,0)
\psline[linewidth=1pt,linecolor=black]{->}(0,-5)(0,6)
\psline[linewidth=1.5pt,linecolor=red]{->}(5,2)(4.9,1.9)(4.8,2.1)(4.7,2.3)(4.6,2.7)(4.8,2.4)(4.5,2.2)(4.3,2.1)(4.1,2)(3.87,1.7)(3.86,2.)(4.,1.8)(3.7,1.5)(3.5,1.3)(3.3,1.4)(3.2,1)(3,1.2)(3.2,0.7)(2.8,0.5)
\psdots[linecolor=red](5,2)
\psdots(3,2.5)
\rput(3,2.7){$x_1$}
\psdots(1,-1.5)
\rput(1,-1.3){$x_2$}
\psdots(-4,4.5)
\rput(-4,4.7){$x_3$}
\psdots(-1,0)
\rput(-1,0.2){$x_4$}
\psdots(-5,-3.5)
\rput(-5,-3.3){$x_5$}
\end{pspicture}
\caption{An Ornstein-Uhlenbeck particle in a random billiard}
\end{figure}
\end{center}

Existence and non explosion of the process, which is especially relevant for $N=+\infty$, will be discussed in a second paper. The process lives in
\begin{equation}
  D = \R^d - \{x \, ; \, |x - x_i|<r_i \, \textrm{ for some } i \} \, ,
\end{equation}
that is, we have removed a collection of non overlapping balls (or more generally non overlapping obstacles).

It is easily seen that the process admits an unique invariant (actually symmetric) probability
measure $\mu_{\lambda,r}$, which is just the Gaussian measure restricted to $D$, i.e.
\begin{equation}\label{eqinvR}
\mu_{\lambda,r}(dx) = Z_{\lambda,r}^{-1} \, \BBone_D(x) \, e^{- \, \lambda \, |x|^2} \, dx \, ,
\end{equation}
where $Z_{\lambda,r}$ is of course a normalizing constant. Hence the process is positive recurrent.

The question now is to obtain estimates for the rate of convergence to equilibrium.
\smallskip

When the number of obstacles $N$ is finite, one can see, using Down, Meyn and Tweedie
results \cite{DMT} and some regularity results for the process following \cite{cat86,cat87},
that the process is exponentially ergodic. It follows from \cite{BCG} theorem 2.1, that
$\mu_{\lambda,r}$ satisfies some Poincar\'e inequality, i.e. for all smooth $f$ (defined on
the whole $\R^d$)
\begin{equation}\label{eqpoinc}
\Var_{\mu_{\lambda,r}}(f) \, \leq \, C_P(\lambda, \mathcal X,r) \, \int \, |\nabla f|^2 \,
\dmu_{\lambda,r} \, .
\end{equation}
But the above method furnishes an horrible (and not really explicit) bound for the Poincar\'e constant $C_P(\lambda,
\mathcal X,r)$. Our aim will thus be to obtain reasonable upper and lower bounds for the Poincar\'e
constant, and to look at the case of infinitely many obstacles, for which the finiteness of the Poincar\'e constant is not even clear. This paper will however focus on the case where there is only one obstacle. An infinite number of particles will be treated separately.
\medskip

Part of the title of the paper is taken from a paper by Lieb et altri \cite{Liebetall} which
is one of the very few papers dealing with Poincar\'e inequality in a sub-domain. Of course,
one cannot get any general result due to the fact that one can always remove an, as small as
we want, subset disconnecting the whole space; so that the remaining sub-domain cannot satisfy
some Poincar\'e inequality. Hence doing this breaks the ergodicity of the process. 

The method used in \cite{Liebetall} relies on the extension of functions defined in $D$ to the whole space. But the inequality they obtain, involves the energy of this extension (including the part inside $D^c$), so that it is not useful to get a quantitative rate of convergence for our process. We have tried to adapt the approach conserving the ``good'' energy, but the estimates obtained are worse than the one we will give below, so that we will not present this approach here. Note also that as we aim at proving a Poincar\'e energy for the state domain of the process, we obtain in a sense version (without additional vector fields) of Theorem 3/Corollary1 in \cite{Liebetall} without a local (obstacle) correction term.
\medskip

Another inspiration for this problem came from the hard balls packing problem, studied from a
``Metropolis'' point of view by the second author in \cite{CFR} (also see Diaconis, Lebeau and Michel \cite{DLM}).
The sub-domain seems to be more complicated, that is why we tried first with a priori simpler obstacles.
\smallskip

Let us explain the simplest case, that enters the framework of the present paper. Consider two hard balls of radius $R=r/2$ in $\mathbb R^d$. Their centers $z_1$ and $z_2$ move randomly, driven by two independent Brownian motions and they are  attracted each other by some linear force. Collisions between the two spheres are supposed to be elastic. If we look at the vector $z=z_1 - z_2$ describing the relative positions of both balls, $z$ is driven by equation \eqref{eqOUR} with $N=1$ and $x_1=0$. Stabilization to equilibrium for large values of $\lambda$ is thus of major importance. The case of three or more (ideally an infinite number of balls) introduces similar problems but for domains $D$ which are more subtle. When $\lambda$ goes to infinity, the invariant measure (the gaussian measure in $D$) converges to the uniform measure on the sets of minimal energy, i.e. will describe the configurations of the optimal packing problem.
\medskip

Another interest is the possible connection with random media problems. Our problem is clearly
related to the second eigenvalue problem (with Neumann condition at the boundary of $D$),
while the book \cite{ASol} dealt with Dirichlet boundary condition (obstacles becoming traps). Presumably also interesting for this point of view, would be to look at the non positive recurrent situation, i.e. replace the Ornstein-Uhlenbeck process by a Brownian motion.
\medskip

Finally, this process can be used as a model for crowds movements when the goal is to slow down the arrivals of people at an exit gate (here the origin) by introducing well chosen obstacles. We shall discuss all this in more details in the second paper. 
\bigskip

In the present paper, as previously emphasized, we shall focus on the case $N=1$, i.e. a single hard obstacle. 

Surprisingly enough (or not) the case of one hard obstacle already contains non trivial features. Our goal is thus to give bounds (as explicit as possible) for the Poincar\'e constant of the gaussian measure restricted to $D$ i.e. the complement of the obstacle. We focus here on the Poincar\'e inequality as it enables us to get the first (non naive) quantitative estimates on the speed of convergence to equilibrium for this process, but other functional inequalities (logarithmic Sobolev inequality, transportation inequality, ...) could be equally considered and the techniques developed here could also prove useful in these cases (for examples Lyapunov techniques have been introduced to study Super Poincar\'e inequalities in \cite{CGWW}, including logarithmic Sobolev inequalities). This will be studied in a following paper.
\medskip

Let us now present notations that we be used throughout the paper. For simplicity we shall write $x_1=y$, $r_1=r$, so that $$\mu_{\lambda,r}(dx) = Z_{\lambda,r,y}^{-1} \,
\BBone_{|x-y|>r} \, e^{-\lambda \, |x|^2} \, dx \, .$$ Once again $Z_{\lambda,r,y}$ denotes the normalizing constant. Translating the measure, we see that
$\mu_{\lambda,r}$ has the same Poincar\'e constant as $$\nu^y_{\lambda,r}(dx) =
Z_{\lambda,r,y}^{-1} \, \BBone_{|x|>r} \, e^{-\lambda \, |x+y|^2} \, dx \, .$$ We also write $C_P(\lambda,y,r)$ for the value of the Poincar\'e constant. It is easily seen, thanks to homogeneity, that
\begin{equation}\label{eqhomo}
C_P(\lambda,y,r) = \frac 1\lambda \, C_P(1,y\sqrt \lambda,r\sqrt \lambda) \, .
\end{equation}
Note that \eqref{eqobstac} is
satisfied for $\mathcal X \, \sqrt \lambda$ and $r\sqrt \lambda$. Hence we have one degree of freedom in the use of all parameters.
\medskip

We have the following conjecture:
\begin{conj}
 there exists an universal constant $C_+$ such that, 
\begin{equation}\label{eqconj}
 \textrm{ for all $y$ and all $r$, } \quad C_P(1,y,r) \leq C_+\left(1+\frac{r^2}{d}\right) \, .
\end{equation} 
\end{conj}
Of course, one can use homogeneity  \eqref{eqhomo} for an extension to $\lambda \neq 1$. This conjecture presents different aspects: first the Poincar\'e constant is independent of the position of the obstacle, secondly the Poincar\'e constant is no more independent of the dimension (contrary to the whole space case for the Gaussian measure) when the size of the obstacle is fixed.\\
\medskip

As will be seen in the sequel, we will use various techniques to tackle this problem: convexity,  perturbation, decomposition of variance, Lyapunov function, isoperimetry ; techniques which are commonly used to derivate a Poincar\'e inequality (in the multidimensional case). Let us give a flavor of the results tending to this conjecture we prove here.\\
\medskip

 {\bf Spherical obstacle: }\\\smallskip
\begin{center} 
\begin{minipage}[c]{13cm}
 \begin{description}
\item[Fact 1] When the obstacle is centered at the origin, the conjecture is true.
\item[Fact 2] The Poincar\'e constant may be bounded above independently of the position of the obstacle.
\item[Fact 3] The Poincar\'e constant grows at least linearly with the radius of the obstacle.
\item[Fact 4] In dimension 2 (at least), if the obstacle is far enough from the origin, then the conjecture is true.
\item[Fact 5] If the obstacle is small w.r.t. the dimension, the conjecture is true.
\end{description}
\end{minipage}
\end{center}

\medskip

{\bf Other Geometries:}\\\smallskip

\begin{center}
\begin{minipage}[c]{13cm}
\begin{description}
\item[Fact 6] With squared obstacle, a phase transition occurs: if the obstacle is small, Poincar\'e constant is bounded, whence if the obstacle is large the Poincar\'e constant explodes with the size obstacle.
\item[Fact 7] One can build traps so that moving the trap at infinity makes the Poincar\'e constant explodes.
\end{description}
\end{minipage}
\end{center}

\bigskip

Let us explain more precisely the plan and contents of the paper.

\medskip

In section \ref{seccentre} we prove that the conjecture \eqref{eqconj} holds when $y=0$. This is done by using spherical symmetry and arguments due to Bobkov in the logconcave case. In particular for $r=0$ we recover up to the constants the Gaussian Poincar\'e constant and when $r\to +\infty$ we recover the Poincar\'e constant of the uniform measure on the sphere of radius $r$, as expected since it is close to the restriction of the Gaussian to $D$. We also obtain a lower bound which is similar to the upper bound, thus proving our conjecture.

Once the ball is no more centered at the origin, as $r\to +\infty$ the measure $\mu_{\lambda,r}$ is close to the Dirac mass $\delta_{y_r}$ where $y_r$ is the (unique) point of the sphere $|x-y|=r$ with minimal distance to the origin. One can thus expect that the Poincar\'e constant is, at least, bounded above independently of $r$. We shall see that this is not the case. Similarly, for a given $r$, we can expect that the Poincar\'e constant is close to the one of the gaussian measure, as $|y| \to +\infty$. We shall see that, replacing balls by hypercubes, this is not the case too.

Using quasi-invariance by translation of the gaussian measure and (somewhat naive) perturbation argument, we obtain some upper bound for the Poincar\'e constant in full generality, that except for specific conditions on $y$ and $r$ (see proposition \ref{propunobys}) is far from the expected one. This is the aim of section \ref{sec perturb}.

When $d\geq 3$ in section \ref{secdecomp} we use the decomposition of variance method to obtain upper bounds. Indeed (up to a rotation) assuming that $y=(a,0)$ for some $a>0$, the conditional distribution of $\mu$ knowing $x_1$ is either the $d-1$ gaussian measure or the $d-1$ gaussian measure restricted to the exterior of a $d-1$ dimensional ball centered at the origin, for which we may apply section \ref{seccentre}. The difficulty is then to control the Poincar\'e constant of the first marginal of $\mu$, and it is at this point that we are not able to recover  fully\eqref{eqconj}. However if the ball is far enough from the origin we obtain an almost satisfactory upper bound.

The main defaults of the previous methods is that they do not extend to the (general) case of more than one obstacle. That is why in section \ref{secunoby2} we develop a ``local'' Lyapunov method (in the spirit of \cite{BBCG}) around the obstacle. As in recent works (\cite{BHW,AKM}) the difficulty is then to piece together the Lyapunov functions we may build near the obstacle and far from the obstacle and the origin. This yields a restriction to small sizes, i.e. $r\sqrt \lambda \leq \frac 12 \sqrt{(d-1)/2}$.

Sections \ref{subseclower}, \ref{secunoby3} and \ref{subsecisopsquare} are devoted to obtain lower bounds. In section \ref{subseclower} we use the relationship between exponential moments of hitting times and the Poincar\'e constant recently described in \cite{CGZ}. In section \ref{secunoby3} we replace euclidian balls by hypercubes. In this situation, for large $r$'s, the Poincar\'e constant is bounded below by some $c(d) e^{b \, r^2}$ for some $b>0$. Hence the situation is drastically different from the case of ``round'' balls. The stochastic interpretation of this phenomenon is that, when starting in the shadow of the obstacle, the Ornstein-Uhlenbeck process is sticked behind the obstacle for a very long time, while on can expect that it slides on the boundary in the spherical situation. 
We investigate further this latter property in section \ref{secunoby4} for $d=2$. Section \ref{subsecisopsquare} deals with lower bounds using this time an isoperimetric approach. Actually we obtain some interesting exploding (as $r\to +\infty$) lower bound for $C_P(1,y,r)$.

\bigskip

\textbf{Acknowledgments.} \quad This work started during a wonderful stay of the second author at the Newton Institute in Cambridge during the spring 2011. P. Cattiaux wishes to heartily thank the organizers of the program ``Discrete Analysis'' and specially F. Barthe for his invitation. \\
This project has been supported by the ANR STAB, (http://math.univ-lyon1.fr/wikis/Stab/).

\bigskip


\section{The case of a centered ball, i.e. $y=0$.}\label{seccentre}

Assume $y=0$. In this case $\mu_{\lambda,r} =
\nu^0_{\lambda,r}$ is spherically symmetric. Though it is not log-concave, its radial part, proportional to
$$\BBone_{\rho >r} \, \rho^{d-1} \, e^{- \lambda \, \rho^2}$$ is log concave in $\rho$ so that we may
use the results in \cite{bobsphere}, yielding
\begin{proposition}\label{propunobcentre}
When $y=0$, the measure $\mu_{\lambda,r}$ satisfies a Poincar\'e inequality
\eqref{eqpoinc} with $$\frac 12 \, \left(\frac{1}{2\lambda} + \frac{r^2}{d}\right) \, \leq
\max\left(\frac{1}{2\lambda} \, , \,  \frac{r^2}{d}\right)\leq \, C_P(\lambda,0,r) \,
\leq \frac{1}{\lambda} + \frac{r^2}{d} \, .$$
\end{proposition}
\begin{proof}
For the upper bound, the only thing to do in view of \cite{bobsphere} is to estimate $\mathbb
E(\xi^2)$ where $\xi$ is a random variable on $\mathbb R^+$ with density 
\begin{equation}\label{eqpol}
\rho \mapsto
A_\lambda^{-1} \, \BBone_{\rho
>r} \, \rho^{d-1} \, e^{- \lambda \, \rho^2} \, .
\end{equation}
 But $$A_\lambda = \int_r^{+\infty} \, \rho^{d-1} \,
e^{- \lambda \, \rho^2} \,d\rho \, \geq \, r^{d-2} \, \int_r^{+\infty} \, \rho \, e^{- \lambda
\, \rho^2} \,d\rho \, = \, \frac{r^{d-2} \, e^{-\lambda \, r^2}}{2 \lambda} \, .$$ A simple
integration by parts yields
$$ \E(\xi^2) = \frac{d}{2 \lambda} + \frac{r^d \, e^{-\lambda \, r^2}}{2 \,
\lambda \, A_\lambda} \, \leq \, \frac{d}{2 \lambda} + r^2 \, .$$ The main result in
\cite{bobsphere} says that $$C_P(\lambda,0,r) \leq \frac{13}{d} \, \E(\xi^2) \, ,$$
hence the result with a constant $13$.

Instead of directly using Bobkov's result, one can look more carefully at its proof. The first part of this proof consists in establishing a bound for the Poincar\'e constant of the law given by \eqref{eqpol}. Here, again, we may apply Bakry-Emery criterion (which holds true on an interval), which furnishes $1/(2\lambda)$. The second step uses the Poincar\'e constant of the uniform measure on the unit sphere, i.e.  $1/d$, times the previous bound for $ \E(\xi^2)$. Finally these two bounds have to be summed up, yielding the result.
\smallskip

For the lower bound it is enough to consider the function $f(z)=\sum_{j=1}^d \, z_j$. Indeed,
the energy of $f$ is equal to $d$. Furthermore on one hand
$$\Var_{\mu_{\lambda,r}}(f) = \frac{\int_r^{+\infty} \, \rho^{d+1} \,
e^{- \lambda \, \rho^2} \,d\rho}{\int_r^{+\infty} \, \rho^{d-1} \, e^{- \lambda \, \rho^2}
\,d\rho} \geq r^2 \, ,$$ while on the other hand, an integration by parts shows that
$$\Var_{\mu_{\lambda,r}}(f) = \frac{d}{2 \lambda} \, + \, \frac{r^d \, e^{- \lambda r^2}}{2\lambda \,
 \int_r^{+\infty} \, \rho^{d-1} \, e^{- \lambda \, \rho^2} \,d\rho} \geq \frac{d}{2 \lambda} \, $$
 yielding the lower bound since the maximum is larger than the half sum.
\end{proof}
This result is satisfactory since we obtain the good order. Notice
that when $r$ goes to 0 we recover (up to some universal constant) the gaussian Poincar\'e
constant, and when $\lambda$ goes to $+\infty$ we recover (up to some universal constant) the
Poincar\'e constant of the uniform measure on the sphere $r S^{d-1}$ which is the limiting
measure of $\mu_{\lambda,r}$. Also notice that the obstacle is really an obstacle since the
Poincar\'e constant is larger than the gaussian one.
\smallskip

\begin{remark}\label{remshell} \quad It is immediate that the same upper bound is true (with the
same proof) for $\nu^0_{{\lambda,r},R}(dx) = Z_{{\lambda,r},R}^{-1} \, \BBone_{R>|x|>r} \,
e^{-\lambda \, |x|^2}$ i.e. the gaussian measure restricted to a spherical shell
$\{R>|x|>r\}$. For the lower bound some extra work is necessary.
\end{remark}
\bigskip


\section{A first estimate for a general $y$ using perturbation.}\label{sec perturb}

An intuitive idea to get estimates on the Poincar\'e constant relies on the Lyapunov function method developed in \cite{BBCG} which requires a local Poincar\'e inequality usually derived from Holley-Stroock perturbation's argument. To be more precise, let us introduce the natural
generator for $\nu^y_{\lambda,r}$ is $$L_y=\frac 12 \, \Delta - \lambda \, \langle x+y,\nabla
\rangle \, .$$ If we consider the function $x \mapsto h(x)=|y+x|^2$ we see that
$$L_y h(x) \, = d - 2\lambda |x+y|^2 \leq \, - \, \lambda \, h(x) \quad \textrm{ if } \quad |x|\geq |y| +
(d/\lambda)^{1/2} \, .$$ So we can use the method in \cite{BBCG}. Consider, for $\varepsilon >0$,  the ball
$$U=B\left(0,\left(|y| + (d/\lambda)^{1/2}\right)\vee (r+\varepsilon)\right) \, .$$ $h$ is a Lyapunov function
satisfying $$L_y h \leq - \lambda h + d \, \BBone_U \, .$$ Since $U^c$ does not intersect the
obstacle $B(0,r)$, we may follow \cite{CGZ} and obtain that
$$C_P(\nu^y_{\lambda,r}) \, \leq \, \frac{4}{\lambda} + \, \left(\frac 4\lambda+ 2\right) \, C_P(\nu_{\lambda,r},U+1)  \,
 ,$$ where $C_P(\nu_{\lambda,r},U+1)$ is the Poincar\'e constant of the
measure $\nu^y_{\lambda,r}$ restricted to the shell $$S = \left\{r<|x|<1+\left(\left(|y| +
(d/\lambda)^{1/2}\right)\vee (r+\varepsilon)\right)\right\} \, .$$ Actually since $h$ may vanish, we
first have to work with $h+\eta$ for some small $\eta$ (and small changes in the constants)
and then let $\eta$ go to $0$ for the dust to settle.
\smallskip

Now we apply Holley-Stroock perturbation argument. Indeed $$\nu^y_{\lambda,r}(dx) =
C(y,\lambda) \, e^{-2 \lambda \, \langle x,y\rangle} \, \nu^0_{\lambda,r}(dx)$$ for some
constant $C(y,\lambda)$. In restriction to the shell $S$, it is thus a logarithmically bounded
perturbation of $\nu^0_{\lambda,r}$ with a logarithmic oscillation less than
$$4 \lambda \, |y| \, \left(1+\left(\left(|y| + ( d/\lambda)^{1/2}\right)\vee (r+\varepsilon)\right)\right)$$ so that we have obtained
 $$C_P(\lambda,y,r) \leq \frac {4}{\lambda} +  \, \left(2+\frac 4\lambda\right)
 \, \left(\frac{1}{\lambda} + \frac{r^2}{d}\right) \, e^{4 \lambda \, |y| \,
\left(1+\left(\left(|y| + (d/\lambda)^{1/2}\right)\vee (r+\varepsilon)\right) \right)} \, .$$ The previous
bound is bad for small $\lambda's$ but one can use the homogeneity property \eqref{eqhomo},
and finally, letting $\varepsilon$ go to $0$
\begin{proposition}\label{propunoby}
For a general $y$, the measure $\mu_{\lambda,r}$ satisfies a Poincar\'e inequality
\eqref{eqpoinc} with $$C_P(\lambda,y,r) \leq \, \frac 2\lambda \, \left(2+3 \, \left(1 + \frac{r^2 \, \lambda}{d}\right) \, e^{4 \sqrt\lambda \, |y| \, \left(1+\left(|y|\sqrt \lambda + d^{1/2}\right)\vee r
\sqrt \lambda\right)}\right) \, .$$
\end{proposition}

The previous result is not satisfactory for large values of $|y|$, $r$ or $\lambda$.
In addition it is not possible to extend the method to more than one obstacle. Finally we have some extra dimension dependence when $y=0$ due to the exponential term. Our aim will now be to improve this estimate. 
\medskip

Another possible way, in order to evaluate the Poincar\'e constant, is to write, for $$g=f -
\frac{ \int \, f(x) \, e^{- \lambda \langle x,y\rangle} \, \nu_{\lambda,r}^0(dx)}{\int \, e^{-
\lambda \langle x,y\rangle} \, \nu_{\lambda,r}^0(dx)}  \, , \, \textrm{ so that } \int g(x) \, e^{-
\lambda \langle x,y\rangle} \, 
\nu_{\lambda,r}^0(dx)=0 \, $$
\begin{eqnarray*}
\Var_{\nu_{\lambda,r}^y}(f) & \leq & \int \, g^2 \, d\nu_{\lambda,r}^y = C(\lambda,y,r) \,
\int \left(g \, e^{-\lambda \langle x,y\rangle}\right)^2 \, d\nu_{\lambda,r}^0\\ & \leq &
C(\lambda,y,r) \,  C_P(\lambda,{0},r) \, \int \, \left|\nabla \left(g \, e^{-\lambda \langle
x,y\rangle}\right)\right|^2 \,  d\nu_{\lambda,r}^0\\ & \leq & 2 \, C_P(\lambda,{0},r) \,
\left( \int \, |\nabla g|^2 \, d\nu_{\lambda,r}^y \, + \, \lambda^2 \, |y|^2 \, \int \, g^2 \,
d\nu_{\lambda,r}^y\right) \, .
\end{eqnarray*}
It follows first that, provided $2 \, C_P(\lambda,{0},r) \, \lambda^2 \, |y|^2 \leq \frac 12$,
$$\int  \, g^2 \, d\nu_{\lambda,r}^y \leq 4 \,  C_P(\lambda,{0},r) \,  \int \, |\nabla g|^2 \,
d\nu_{\lambda,r}^y \, ,$$ and finally

\begin{proposition}\label{propunobys}
If $4 \, \lambda \, |y|^2 \, \left(1 + \frac{r^2 \,
\lambda}{d}\right) \, \leq 1$, the measure $\mu_{\lambda,r}$ satisfies a Poincar\'e inequality
\eqref{eqpoinc} with
$$C_P(\lambda,y,r) \leq \, 4 \, \left(\frac{1}{\lambda} + \frac{r^2}{d}\right) \, .$$
\end{proposition}

One can note that under the condition  $4 \, \lambda \, |y|^2 \, \left(1 + \frac{r^2
\, \lambda}{d}\right) \, \leq 1$, Proposition \ref{propunoby} and Proposition \ref{propunobys}
yield, up to some dimension dependent constant, similar bounds. Of course the first proposition is more general.
\bigskip


\section{A second estimate for a general $y$ using decomposition of variance.}\label{secdecomp}

In this section for simplicity we will first assume that $\lambda=1$, and second that $d\geq 3$.

Using rotation invariance we may also assume that $y=(a,0)$ for some $a\in \R^+$, $0$ being the null vector of $\mathbb R^{d-1}$. Since we shall use an induction procedure, we add the index $d$ in our notation, and suppress $\lambda$ since it is equal to $1$. So, writing $x=(u,\bar x) \in \R\times \R^{d-1}$, $$\mu_{d,r}(du,d\bar x)= \nu^0_{d-1,R(u)}(d\bar x) \, \mu_1(du) \, ,$$ where $\nu^0_{d-1,R(u)}(d\bar x)$ is the $d-1$ dimensional gaussian measure restricted to $B^c(0,R(u))$ as in section \ref{seccentre} with $R(u)=\sqrt{\left(\left(r^2 - (u-a)^2\right)_+\right)}$ and $\mu_1$ is the first marginal of $\mu_{d,r}$ given by $$\mu_1(du) = \frac{\gamma_{d-1}(B^c(0,R(u)))}{\gamma_d(B^c(y,r))} \, \, \gamma_1(du) \, ,$$ $\gamma_n$ denoting the $n$ dimensional gaussian measure $c_n \, e^{-|x|^2} \, dx$.

The standard decomposition of variance tells us that for a nice $f$,
\begin{equation}\label{eqdecompvar}
\Var_{\mu_{d,r}}(f) = \int \, \left(\Var_{\nu^0_{d-1,R(u)}}(f)\right) \, \mu_1(du) + \Var_{\mu_1}(\bar f) \, ,
\end{equation}
where $$\bar f(u) = \int \, f(u,\bar x) \, \nu^0_{d-1,R(u)}(d \bar x) \, .$$
According to Proposition \ref{propunobcentre}, on one hand, it holds for all $u$, 
\begin{equation}\label{eqpoinccondit}
\Var_{\nu^0_{d-1,R(u)}}(f) \leq \left(1+\frac{(r^2 - (u-a)^2)_+}{d-1}\right) \, \int \, |\nabla_{\bar x} f|^2 \, d\nu^0_{d-1,R(u)} \, ,
\end{equation} 
so that 
\begin{equation}\label{eqcondit1}
\int \, \left(\Var_{\nu^0_{d-1,R(u)}}(f)\right) \, \mu_1(du) \leq \left(1+ \frac{r^2}{d-1}\right) \, \int \, |\nabla_{\bar x} f|^2 \, d\mu_{d,r} \, .
\end{equation}
On the other hand, $\mu_1$ is a logarithmically bounded perturbation of $\gamma_1$ hence satisfies some Poincar\'e inequality so that
\begin{equation}\label{eqcondit2}
\Var_{\mu_1}(\bar f) \leq C_1 \, \int \, \left|\frac{d\bar f}{du}\right|^2 \, d\mu_1 \, .
\end{equation}
So we have first to get a correct bound for $C_1$, second to understand what $\frac{d\bar f}{du}$ is.
\medskip

\subsection{A bound for $C_1$.}\label{subsecC1} \quad Since $\mu_1$ is defined on the real line, upper and lower bounds for $C_1$ may be obtained by using Muckenhoupt bounds (see \cite{logsob} Theorem 6.2.2). Unfortunately we were not able to obtain the corresponding explicit expression in our situation as $\mu_1$ is not sufficiently explicitly given to use Muckenhoupt criterion. So we shall give various upper bounds using other tools.
\smallskip

The usual Holley-Stroock perturbation argument combined with the Poincar\'e inequality for $\gamma_1$ imply that 
\begin{equation}\label{eqC1a}
C_1 \leq \frac 12 \, \frac{\sup_u \, \{\gamma_{d-1}(B^c(0,R(u)))\}}{\inf_u \, \{\gamma_{d-1}(B^c(0,R(u)))\}} \leq  \frac 12 \, \frac{\int_0^{+\infty} \, \rho^{d-2} \, e^{-\rho^2} \, d\rho}{\int_r^{+\infty} \, \rho^{d-2} \, e^{-\rho^2} \, d\rho} = \frac 12 \, \left(1+ \frac{\int_0^r \, \rho^{d-2} \, e^{-\rho^2} \, d\rho}{\int_r^{+\infty} \, \rho^{d-2} \, e^{-\rho^2} \, d\rho}\right) \, .
\end{equation}
Using the first inequality and the usual lower bound for the denominator, it follows that $$ \textrm{for all $r>0$, } \quad C_1 \leq \pi^{(d-2)/2} \, \frac{e^{r^2}}{r^{d-3}} \, .$$
The function $\rho \mapsto \rho^{d-2} \, e^{-\rho^2}$ increases up to its maximal value which is attained for $\rho^2=(d-2)/2$ and then decreases to $0$. It follows, using the second form of the inequality \eqref{eqC1a} that 
\begin{itemize}
\item \quad if $r\leq \sqrt{\frac{d-2}{2}}$ we have $C_1 \leq \frac 12 +r^2 $, while 
\item \quad if $r\geq \sqrt{\frac{d-2}{2}}$ we have $$C_1 \leq \frac 12 \, + \left(\frac{d-2}{2}\right)^{\frac{d-2}{2}} \, e^{- \frac{d-2}{2}} \, \frac{e^{r^2}}{r^{d-4}} \, .$$
\end{itemize}
These bounds are quite bad for large $r$'s but does not depend on $y$.
\smallskip

Why is it bad ? First for $a=0$ (corresponding to the situation of section \ref{seccentre}) we know that $C_1 \leq 1 + \frac{r^2}{d}$ according to Proposition \ref{propunobcentre} applied to functions depending on $x_1$. Actually the calculations we have done in the proof of proposition \ref{propunobcentre}, are unchanged for $f(z)=z_1$, so that it is immediately seen that $C_1 \geq \max(\frac 12,\frac{r^2}{d})$.

Intuitively the case $a=0$ is the worst one, though we have no proof of this. We can nevertheless give some hints. 
\smallskip

The natural generator associated to $\mu_1$ is 
\begin{eqnarray*}
L_1 &=& \frac{d^2}{du^2} - \left(u -\frac{d}{du} \log( \gamma_{d-1}(B^c(0,R(u))))\right) \, \frac{d}{du} \\ &=& \frac{d^2}{du^2} - u \, \frac{d}{du} + \, \frac{(u-a) \, (R(u))^{d-3} \, e^{-R^2(u)}}{\int_{R(u)}^{+\infty} \, \rho^{d-2} \, e^{-\rho^2}  \, d\rho} \, \, \BBone_{|u-a|\leq r} \, \frac{d}{du} \, . 
\end{eqnarray*}
The additional drift term behaves badly for $a\leq u \leq a+r$, since in this case it is larger than $-u$, while for $u\leq a$ it is smaller. In stochastic terms it means that one can compare the induced process with the Ornstein-Uhlenbeck process except possibly for $a\leq u \leq a+r$. In analytic terms let us look for a Lyapunov function for $L_1$. As for the O-U generator the simplest one is $g(u)=u \mapsto u^2$ for which $$L_1g \leq 2 - 4u^2 + \, 4 u(u-a) \, \, \BBone_{a\leq u\leq a+r} \, .$$ Remember that $a\geq r$ so that $-au \leq \frac 12 \, u^2$. It follows
\begin{equation}\label{eqloin}
\textrm{ provided $a\geq r$, } \quad L_1 g \leq 2 - 2g \, .
\end{equation}
For $|u|\geq 2$ we then have $L_1g(u)\leq -  \, g(u)$, so that $g$ is a Lyapunov function outside the interval $[-\sqrt 2, \sqrt 2]$ and the restriction of $\mu_1$ to this interval coincides (up to the constants) with the gaussian law $\gamma_1$ hence satisfies a Poincar\'e inequality with constant $\frac 12$ on this interval. According to the results in \cite{BBCG} we recalled in the previous section, we thus have that $C_1$ is bounded above by some universal constant $c$.

We may gather our results
\begin{lemma}\label{lemC1}
The following upper bound holds for $C_1$ :
\begin{enumerate}
\item[(1)] \quad (small obstacle) if $r\leq \sqrt{\frac{d-2}{2}}$ we have $C_1 \leq \frac 12 +r^2 $,
\item[(2)] \quad (far obstacle) if $|y|>r + \sqrt 2$, $C_1\leq c$ for some universal constant $c$,
\item[(3)] \quad (centered obstacle) if $y=0$, $C_1 \leq 1 + \frac{r^2}{d}$,
\item[(4)] \quad in all other cases, there exists $c(d)$ such that $C_1\leq  \, c(d) \, \frac{e^{r^2}}{r^{d-3}} \, .$
\end{enumerate}
We conjecture that actually $C_1 \leq C (1+r^2)$ for some universal constant $C$.
\end{lemma} 
\medskip

\subsection{Controlling $\frac{d\bar f}{du}$.}\label{subsecIPP} \quad It remains to understand what $\frac{d\bar f}{du}$ is and to compute the integral of its square against $\mu_1$.

Recall that  $$\bar f(u) = \int \, f(u,\bar x) \, \nu^0_{d-1,R(u)}(d \bar x) \, .$$ Hence
\begin{eqnarray*}
\bar f(u) &=& \BBone_{|u-a|> r} \, \int \, f(u,\bar x) \, \nu^0_{d-1,0}(d \bar x) \\ && + \, \BBone_{|u-a|\leq r} \, \int_{\mathbb S^{d-2}} \int_{R(u)}^{+\infty} \, f(u,\rho \, \theta) \, \frac{\rho^{d-2} \, e^{-\rho^2}}{c(d) \, \gamma_{d-1}(B^c(0,R(u)))} \, d\rho \, d\theta \, ,
\end{eqnarray*}
where $d\theta$ is the non-normalized surface measure on the unit sphere $\mathbb S^{d-2}$ and $c(d)$ the normalization constant for the gaussian measure. Hence, for $|u-a|\neq r$ we have
\begin{eqnarray*}
\frac{d}{du} \bar f(u) &=& \int \, \frac{\partial f}{\partial x_1}(u,\bar x) \, \nu^0_{d-1,R(u)}(d \bar x) \\ && - \, \BBone_{|u-a|\leq r} \, \int \, f(u,\bar x) \, \BBone_{|\bar x|>R(u)} \, \frac{\frac{d}{du} \, \left( \gamma_{d-1}(B^c(0,R(u)))\right)}{\gamma^2_{d-1}(B^c(0,R(u)))} \, \gamma_{d-1}(d\bar x) \\ && - \, \BBone_{|u-a|\leq r} \,  \frac{R'(u) \, R^{d-2}(u) \, e^{-R^2(u)}}{c(d) \, \gamma_{d-1}(B^c(0,R(u)))} \, \int_{\mathbb S^{d-2}} \, f(u, R(u) \, \theta) \, d\theta \, .
\end{eqnarray*}
Notice that if $f$ only depends on $u$, $\bar f=f$ so that $$\frac{d}{du} \bar f(u) = \frac{\partial f}{\partial x_1}(u) = \int \, \frac{\partial f}{\partial x_1}(u) \, \nu^0_{d-1,R(u)}(d \bar x) \, ,$$ and thus the sum of the two remaining terms is equal to $0$. Hence in computing the sum of the two last terms, we may replace $f$ by $f - \int f(u,\bar x) \, \nu^0_{d-1,R(u)}(d \bar x)$ or if one prefers, we may assume that the latter $\int f(u,\bar x) \, \nu^0_{d-1,R(u)}(d \bar x)$ vanishes. Observe that this change will not affect the gradient in the $\bar x$ direction.

Assuming this, the second term becomes $$- \, \BBone_{|u-a|\leq r} \, \frac{\frac{d}{du} \, \left( \gamma_{d-1}(B^c(0,R(u)))\right)}{\gamma_{d-1}(B^c(0,R(u)))} \, \int \, f(u,\bar x) \, \nu^0_{d-1,R(u)}(d \bar x) \, = \, 0 \, .$$ 
\smallskip

We thus have (using Cauchy-Schwarz inequality) and a scale change
\begin{eqnarray*}
\int \, \left|\frac{d\bar f}{du}\right|^2 \, d\mu_1 &\leq& 2 \, \int \, \left|\frac{\partial f}{\partial x_1}\right|^2 (u, \bar x) \, \mu_{d,r}(du,d\bar x) \\ && + \, 2  \, \int \left(\BBone_{|u-a|\leq r} \,  \frac{R'(u) \, e^{-R^2(u)}}{c(d) \, \gamma_{d-1}(B^c(0,R(u)))} \, \int_{\mathbb S^{d-2}(R(u))} \, f(u, \, \theta) \, d\theta \right)^2 \, \mu_1(du) \, . 
\end{eqnarray*}
Our goal is to control the last term using the gradient of $f$. One good way to do it is to use the Green-Riemann formula, in a well adapted form. Indeed, let $V$ be a vector field written as
\begin{equation}\label{eqGR1}
V(\bar x) = - \, \frac{\varphi(|\bar x|)}{|\bar x|^{d-1}} \, \, \bar x \quad \textrm{ where } \varphi(R(u))=R^{d-2}(u) \, .
\end{equation} 
This choice is motivated by the fact that the divergence, $\nabla . (\bar x/|\bar x|^{d-1}) = 0$ on the whole $\R^{d-1}-\{0\}$. 

Of course in what follows we may assume that $R(u)>0$, so that all calculations make sense. The Green-Riemann formula tells us that, denoting $g_u(\bar x)= f(u, \bar x)$, for some well choosen $\phi$
\begin{eqnarray*}
\int_{\mathbb S^{d-2}(R(u))} \, f(u, \, \theta) \, d\theta &=& \int_{\mathbb S^{d-2}(R(u))} \, g_u \,  \langle V,(- \bar x/|\bar x|)\rangle \, d\theta = \int \, \BBone_{|\bar x|\geq R(u)} \, \nabla . (g_uV)(\bar x) \, d\bar x\\ &=& - \, \int \, \BBone_{|\bar x|\geq R(u)} \, \langle \nabla g_u(\bar x),(\bar x/|\bar x|^{d-1})\rangle \, \varphi(|\bar x|) \, d\bar x \\ && - \,  \int \, \BBone_{|\bar x|\geq R(u)} \, g_u(\bar x) \, (\varphi'(|\bar x|)/|\bar x|^{d-1}) \, d\bar x \, .
\end{eqnarray*}
\bigskip
Now we choose $\varphi(s)=R^{d-2}(u) \, e^{R^2(u)} \, e^{-s^2}$ and recall that $R'(u)= -((u-a)/R(u)) \, \BBone_{|u-a|\leq r}$. We have finally obtained
\begin{eqnarray*}
&& \BBone_{|u-a|\leq r} \,  \frac{R'(u) \, e^{-R^2(u)}}{c(d) \, \gamma_{d-1}(B^c(0,R(u)))} \, \int_{\mathbb S^{d-2}(R(u))} \, f(u, \, \theta) \, d\theta \, = \\ && = \BBone_{|u-a|\leq r} \, (u-a) \, R^{d-3}(u) \, \int \, \langle\nabla_{\bar x}f (u,\bar x),(\bar x/|\bar x|^{d-1})\rangle \, \nu^0_{d-1,R(u)}(d \bar x) \, \\ && \, - \, \BBone_{|u-a|\leq r} \, (u-a) \, R^{d-3}(u) \, 2 \, \int \, (f(u,\bar x)/ |\bar x|^{d-3}) \, \nu^0_{d-1,R(u)}(d \bar x) \, .
\end{eqnarray*}
To control the first term we use Cauchy-Schwarz inequality, while for the second one we use Cauchy-Schwarz and the Poincar\'e inequality for $\nu^0_{d-1,R(u)}$, since $\int \, f(u,\bar x) \, \nu^0_{d-1,R(u)}(d \bar x) \, = \, 0$. This yields
\begin{eqnarray*}
\int \, \left|\frac{d\bar f}{du}\right|^2 \, d\mu_1 &\leq& 2 \, \int \, \left|\frac{\partial f}{\partial x_1}\right|^2 (u, \bar x) \, \mu_{d,r}(du,d\bar x)\\ && +  \, 4 \, \int |\nabla_{\bar x}f|^2 \, \mu_{d,r}(du,d\bar x) \, (A_1+ 4 \, A_2)
\end{eqnarray*}
where $$A_1= \int |u-a|^2 \, \BBone_{|u-a|\leq r} \, R^{2d-6}(u) \, \left(\int |\bar x|^{4-2d} \, \nu^0_{d-1,R(u)}(d \bar x)\right) \, \mu_1(du) \, ,$$ and $$A_2= \int |u-a|^2 \, \BBone_{|u-a|\leq r} \, R^{2d-6}(u) \, \left(1 + \frac{R^2(u)}{d-1}\right) \, \left(\int |\bar x|^{6-2d} \, \nu^0_{d-1,R(u)}(d \bar x)\right)  \, \mu_1(du) \, .$$ 
It is immediate (recall that  the support of $\nu^0_{d-1,R(u)}$ is $|\bar x|\geq R(u)$) that $$A_2 \, \leq r^2 \, \left(1+\frac{1}{d-1} \, \int_{R(u)>0} R^2(u) \, \mu_1(du)\right) \, .$$ If $r\leq \beta \sqrt{d-1}$ we thus have $A_2 \leq (1+\beta^2) r^2$. In full generality it holds $A_2 \leq  r^2 \, (1+(r^2/d-1))$. 
\smallskip

This bound can be improved for large $r$'s provided $a$ is large too. Indeed, on $R(u)>0$, $$\mu_1(du) \leq \frac{\gamma_{d-1}(B^c(0,R(u)))}{\gamma_{d}(B^c(0,r))} \, \gamma_1(du)\leq  \, c \, \frac{e^{r^2-u^2}}{r} \, \left(\int_{R(u)}^{+\infty} \, \rho^{d-2} \, e^{-\rho^2} \, d\rho\right) \, du,$$ for some universal constant $c$. Using integration by parts we have, for $z>0$, 
\begin{eqnarray*}
\int_{z}^{+\infty} \, \rho^{d-2} \, e^{-\rho^2} \, d\rho &\leq& \frac 12 \, z^{d-3} \, e^{-z^2} + \frac{d-3}{2} \, \int _{z}^{+\infty} \, \rho^{d-4} \, e^{-\rho^2} \, d\rho\\ &\leq& \frac 12 \, z^{d-3} \, e^{-z^2} + \frac{d-3}{2z^2} \, \int_{z}^{+\infty} \, \rho^{d-2} \, e^{-\rho^2} \, d\rho \, , 
\end{eqnarray*}
so that, provided $z^2 >d-1$, $$\int_{z}^{+\infty} \, \rho^{d-2} \, e^{-\rho^2} \, d\rho \leq \frac{z^2}{2z^2 -(d-3)} \, z^{d-3} \, e^{-z^2} \, \leq \, \frac{z^{d-1}}{2d+1} \, e^{-z^2} \, .$$  To bound $A_2$, we perform the integral on $R(u)\leq \sqrt{d-1}$ and $R(u)> \sqrt{d-1}$, so that using the previous bound we obtain 
\begin{eqnarray*}
A_2 &\leq& r^2 \, \left(2 + c \, \frac{r^d}{(d-1)(2d+1)} \, \int_{R(u)>\sqrt{d-1}} e^{r^2-u^2-R^2(u)} \, du\right) \,\\  &\leq& \, r^2 \, \left(2 + c \, \frac{2 \, r^{d+2}}{(d-1)(2d+1)} \, e^{r^2-(a-r)^2-(d-1)}\right) \, ,
\end{eqnarray*}
provided $a>r$. If $a>(2+\alpha) r$ for some $\alpha >0$ we thus have, $$A_2 \leq  r^2 \, \left(2 + c \, \frac{2 \, r^{d+2}}{(d-1)(2d+1)} \, e^{-\alpha \, r^2}\right) \leq C(\alpha) \, r^2 \, ,$$ for some $C(\alpha)$ that only depends on $\alpha$ (and not on $d$).
\smallskip

Finally we have obtained
\begin{enumerate}
\item \, if for some $\alpha >0$, $r<\alpha \sqrt{d-1}$ or $a>(2+\alpha) r$, $A_2 \leq C(\alpha) \, r^2$ ,
\item \, in all cases $A_2 \leq  r^2 \, (1+(r^2/d-1))$.
\end{enumerate}
\smallskip

The control of $A_1$ is also a little bit delicate. Indeed we have to split the integral in two parts, the first one corresponding to the $u$'s such that $R(u)\geq 1$ (if this set is not empty), the second one to the $u$'s such that $R(u)<1$. Thus we have the following rough bound
$$
A_1 \leq r^2 \left(1 + \int_{0<R(u)\leq 1} \, R^{2d-6}(u) \, \frac{\int_{\rho>R(u)} \rho^{2-d} \, e^{-\rho^2} \, d\rho}{\int_{\rho>R(u)} \rho^{d-2} \, e^{-\rho^2} \, d\rho} \, \mu_1(du)\right)  \, .$$
To bound the second term in the sum, we use, for $d>3$, $$\int_{\rho>R(u)} \rho^{2-d} \, e^{-\rho^2} \, d\rho \leq \int_{\rho>R(u)} \rho^{2-d} \, d\rho = \frac{R^{3-d}(u)}{d-3}$$ and $$\int_{\rho>R(u)} \rho^{d-2} \, e^{-\rho^2} \, d\rho \geq  \frac{R^{d-3}(u)}{2e} \, .$$ Combining these two bounds, we obtain $$A_1 \leq r^2 \left(1 + \int_{0<R(u)\leq 1} \, \frac{2e}{d-3} \, \mu_1(du)\right) \leq r^2 \left(1 + \frac{2e}{d-3}\right) \, ,$$
provided $d>3$.

If $d=3$, we have 
$$
\int_{\rho>R(u)} \rho^{2-d} \, e^{-\rho^2} \, d\rho \leq  \int_{1>\rho>R(u)} \rho^{-1}  \, d\rho + \int_{\rho>1}  \, e^{-\rho^2} \, d\rho \leq \log(1/R(u)) + \sqrt{\pi} \, .$$
It follows
$$A_1 \leq r^2 \left(1 + 2e \, \sqrt{\pi} + \int_{0<R(u)\leq 1} \, 2e \, \log(1/R(u)) \, \mu_1(du)\right) \, .$$ It remains to get an upper bound for $$B_1 = \int_{0<R(u)\leq 1} \, \log(1/R^2(u)) \, \mu_1(du) \, .$$ When $r\leq 1(=(\sqrt{d-1}/\sqrt{2})$, we have for some universal constant $c$ that may vary from line to line,
\begin{eqnarray*}
B_1 &\leq& - \, c \, \int_{a-r}^{a+r} \, \log(r^2-(u-a)^2) \, \frac{\gamma_{2}(B^c(0,R(u)))}{\gamma_{3}(B^c(0,r))} \, e^{-u^2} \, du \\ &\leq& \,  - \, c \, \int_{a-r}^{a+r} \, \log(r^2-(u-a)^2) \, e^{-u^2 -R^2(u)} \, du \\ &\leq& \,  - \, c \, \int_{a-1}^{a+1} \, \log(1-(u-a)^2) \, du \\ &\leq& \, c \, .
\end{eqnarray*}
When $r>1$ the integral splits in two terms 
\begin{eqnarray*}
B_1 &=& - \, c \, \int_{a-r}^{a-\sqrt{r^2-1}} \, \log(r^2-(u-a)^2) \, \frac{e^{-u^2-R^2(u)+r^2}}{r} \, du \\ & & - \, c \, \int^{a+r}_{a+\sqrt{r^2-1}} \, \log(r^2-(u-a)^2) \, \frac{e^{-u^2-R^2(u)+r^2}}{r} \, du \, .
\end{eqnarray*}
Note that, provided $a>2r$, $-u^2-R^2(u)+r^2 = - a(2u-a) \leq - a(a-2r)\leq 0$ in the first integral while $-u^2-R^2(u)+r^2 \leq - a^2 \leq 0$ for all $a$ in the second one. So we have, using the change of variable  $u-a=-r + rv$ (resp. $u-a=r-rv$) and recalling that $c$ may vary but is still universal,
\begin{eqnarray*}
B_1 &\leq& - \, c \, \int_0^{(r-\sqrt{r^2-1})/r} \, \log(r) \, \log(v(2-v)) \, dv  \leq \, c \, \log(r) \, .
\end{eqnarray*}
If we assume that $a>(2+\alpha) r$ for some $\alpha>0$, one can improve the previous bound in $c(\alpha)$ independent of $r$.

Unfortunately, when $0\leq a\leq 2r$ we only obtain $B_1 \leq c \, \log(r) \, e^{a(2r-a)}$.

We have thus obtained
\begin{enumerate}
\item \quad if $d>3$ then $A_1 \leq c \, r^2$, \item \quad for $d=3$, if $r\leq 1$ or $a>(2+\alpha)r$, $A_1 \leq c r^2$, \item \quad for $d=3$, $r>1$ and $a>2r$, $A_1 \leq c \, r^2 \, \log(r)$, \item \quad for $d=3$, $r>1$ and $0<a<2r$, $A_1 \leq c \, r^2 \,(1+ e^{a(2r-a)}) \leq c \, r^2  \, e^{r^2}$.
\end{enumerate}

Gathering together all we have done we have shown

\begin{theorem}\label{thmmanu}
Assume $d \geq 3$. There exists a function $C(r,d)$ such that, for  all $y \in \R^d$, $$C_P(1,y,r) \, \leq \, C(r,d) \, .$$ Furthermore, there exists some universal constant $c$ such that $$C(r,d) \leq \left(1+\frac{r^2}{d-1}\right) + C_1(r) \, \max \left(2 \, ,  C_2(r)\right) \, ,$$ $C_1(r)$ being given in Lemma \ref{lemC1} and $C_2(r)$ satisfying 
\begin{enumerate}
\item \quad if $r\leq \sqrt{(d-1)/2}$ or $|y|>(2+\alpha)r$, $C_2(r) \leq c \, r^2$, \item \quad if $d>3$ or $d=3$, $r\geq 1$ and $|y|>2r$, $C_2(r) \leq c \, r^2\left(1+\frac{r^2}{d-1}\right)$, \item \quad if $d=3$, $r\geq 1$ and $0\leq |y|\leq 2r$, $C_2(r) \leq c \, r^2 \, \max \left(r^2 \, , \, e^{|y|(2r-|y|)}\right)$.
\end{enumerate}
\end{theorem}
\medskip

\begin{remark}\label{remmanu}
The previous theorem is interesting as it shows that the Poincar\'e constant is bounded uniformly in $y$. For far obstacles, meaning $|y|\geq (2+\alpha)r + \sqrt 2$ for some positive $\alpha$, we get an upper bound for the Poincar\'e constant of order $1+r^2$ that matches our conjecture. However, in other cases, even if our conjecture on $C_1$ is true, it would furnish a more pessimistic control in $r$ which behaves like $r^4$ for large $r$'s in most of the cases, while we are expecting something like $r^2$.

The method suffers nevertheless two defaults. First it does not work for $d=2$, in which case the conditioned measure does no more satisfy a Poincar\'e inequality. More important for our purpose, as for the previous section, the method does not extend to more than one obstacle, unless the obstacles have a particular location. 
\end{remark}


\section{Improving the estimate for a general $y$ and small obstacles.}\label{secunoby2}  

In what we did previously we sometimes used Lyapunov functions vanishing in a neighborhood of the
obstacles. Indeed a Lyapunov function (generally) has to belong to the domain of the
generator, in particular its normal derivative (generally) has to vanish on the boundary of
the obstacle. Since it seems that a squared distance is a good candidate it is natural to look
at the geodesic distance in the punctured domain $D$ (see \cite{Alex} and also \cite{Harg} for
small time estimates of the density in this situation). Unless differentiability problems (the
distance is not everywhere $C^2$) it seems that this distance does not yield the appropriate
estimate (calculations being tedious).

Instead of trying to get a ``global'' Lyapunov function, we shall build ``locally'' such
functions.
\smallskip

To be more precise, consider an open  neighborhood (in $D$) $U$ of the obstacle $B(y,r)$ and
some smooth function $\chi$ compactly supported in $B^c(y,r)$ such that $\BBone_{U^c} \leq
\chi \leq 1$. Let $f$ be a smooth function and $m$ be such that $\int \chi \, (f-m) \,
d\mu_{\lambda,r} = 0$. Then
\begin{eqnarray*}
\Var_{\mu_{\lambda,r}}(f)&  \leq & \int_D \, (f-m)^2 d\mu_{\lambda,r} = \int_U \, (f-m)^2
d\mu_{\lambda,r} + \int_{U^c} \, (f-m)^2 d\mu_{\lambda,r} \\ & \leq & \int_U \, (f-m)^2
d\mu_{\lambda,r} + \int_{\mathbb R^d} \, \chi^2 \, (f-m)^2 d\mu_{\lambda,r}\\ & \leq & \int_U
\, (f-m)^2 \, d\mu_{\lambda,r} + \frac{1}{2\lambda} \, \int_{\mathbb R^d} \, |\nabla(\chi \,
(f-m))|^2 \, d\mu_{\lambda,r} \\ & \leq & \int_U \, (f-m)^2 \, d\mu_{\lambda,r} +
\frac{1}{\lambda} \, \int_D \, \left(|\nabla \chi|^2 \, (f-m)^2 \, + \chi^2 \, |\nabla
f|^2\right) \, d\mu_{\lambda,r} \, ,
\end{eqnarray*}
so that
\begin{eqnarray}\label{eqhowtolocalize}
\Var_{\mu_{\lambda,r}}(f) & \leq & \int_D \, (f-m)^2 d\mu_{\lambda,r} \nonumber \\ & \leq &
\left(1+\frac{\parallel \nabla \chi\parallel_\infty^2}{\lambda}\right) \, \int_U \, (f-m)^2 \,
d\mu_{\lambda,r} + \frac 1\lambda \, \int_D \, |\nabla f|^2 \, d\mu_{\lambda,r} \, .
\end{eqnarray}
We thus see that what we have to do is to get some bound for $\int_U \, (f-m)^2 \,
d\mu_{\lambda,r}$ in terms of the energy $\int_D \, |\nabla f|^2 \, d\mu_{\lambda,r}$ for any
smooth $f$ which is exactly what is done by finding a ``local'' Lyapunov function.

\subsection{Two useful lemmas on Lyapunov function method.\\ \\}

We may now present two particularly useful lemmas concerning Lyapunov function method and localization.
Let us begin by the following remark: in the previous derivation assume that for
some $p>1$ and some constant $C$,
\begin{equation}\label{eqpoidsy}
\int_U \, (f-m)^2 d\mu_{\lambda,r} \leq \frac{\lambda}{p \, \parallel \nabla
\chi\parallel_\infty^2} \,  \int_{\mathbb R^d} \, \chi^2 \, (f-m)^2 d\mu_{\lambda,r} \,+ \, C
\, \int_D \, |\nabla f|^2 \, d\mu_{\lambda,r} .
\end{equation}
Then, using the Poincar\'e inequality for the gaussian measure, we have
\begin{eqnarray*}
\int_{\mathbb R^d} \, \chi^2 \, (f-m)^2 d\mu_{\lambda,r} & \leq & \frac{1}{\lambda} \,
\int_{\mathbb R^d}
\, \left(|\nabla \chi|^2 \, (f-m)^2 \, + \chi^2 \, |\nabla f|^2\right) \, d\mu_{\lambda,r} \\
& \leq & \frac{\parallel \nabla \chi\parallel_\infty^2}{\lambda} \, \int_U \, (f-m)^2 \,
d\mu_{\lambda,r} + \frac 1\lambda \, \int_D \, |\nabla f|^2 \, d\mu_{\lambda,r} \\
& \leq & \frac 1p \, \int_{\mathbb R^d} \, \chi^2 \, (f-m)^2 d\mu_{\lambda,r} + \frac 1\lambda
\, (1 + C
\parallel \nabla \chi\parallel_\infty^2) \, \int_D \, |\nabla f|^2 \, d\mu_{\lambda,r}
\end{eqnarray*}
so that $$\int_{\mathbb R^d} \, \chi^2 \, (f-m)^2 d\mu_{\lambda,r} \leq \frac{p}{(p-1)\lambda}
\, (1 + C
\parallel \nabla \chi\parallel_\infty^2) \, \int_D
\, |\nabla f|^2 \, d\mu_{\lambda,r}$$ and using \eqref{eqpoidsy}
$$\int_U \, (f-m)^2 d\mu_{\lambda,r} \leq \left(C + \frac{1}{(p-1) \, \parallel \nabla
\chi\parallel_\infty^2} \, (1 + C
\parallel \nabla \chi\parallel_\infty^2)\right) \, \int_D \, |\nabla f|^2 \, d\mu_{\lambda,r}$$ and finally
\begin{lemma}\label{lempoidsy}
If \eqref{eqpoidsy} holds, $$ \Var_{\mu_{\lambda,r}}(f) \leq \frac{1}{p-1} \, \left(Cp + \frac{1}{\parallel \nabla \chi\parallel_\infty^2} + \frac{p(1+C \parallel \nabla \chi\parallel_\infty^2)}{\lambda}\right)  \, \int_D \, |\nabla f|^2 \,
d\mu_{\lambda,r} \, .$$
\end{lemma}
\bigskip

In the sequel $U$ will be an open ball centered at $y$. Without loss of generality (if
necessary) we may assume that $y=(a,0)$ for some $a\in \R^+$, $0$ being the null vector of
$\mathbb R^{d-1}$. The (non normalized) normal vector field at the boundary of $B(y,r)$, pointing inward $D$, 
is thus $x-y = (x^1-a,\bar x)\in \R \times \R^{d-1}$. We shall denote by $n$ the normalized
inward normal vector field.

Now recall the basic lemma used in \cite{BBCG,CGZ} we state here
in a slightly more general context (actually this lemma is more or less contained in \cite{CGZ} Remark 3.3)

\begin{lemma}\label{lemIPP}
Let $f$ be a smooth function with compact support in $\bar D$ and $W$ a positive smooth
function. Denote by $\mu_{\lambda,r}^S$ the trace (surface measure) on $S(y,r)=\{|x-y|=r\}$ of
$\mu_{\lambda,r}$. Then the following holds
$$\int_D \, \frac{-LW}{W} \, f^2 \, d\mu_{\lambda,r} \, \leq \, \frac 12 \,  \int_D
\,  |\nabla f|^2 \, d\mu_{\lambda,r} \, + \, \frac 12 \, \int_{S(y,r)} \, \frac{\partial
W}{\partial n} \, \frac{f^2}{W} \, d\mu_{\lambda,r}^S \, .$$
\end{lemma}
\begin{proof}
We recall the proof for the sake of completeness. Using the first Green formula we have
(recall that $n$ is pointing inward)
\begin{align*}
\int_D \, \frac{- 2 \, LW}{W} \, f^2 \dmu_{\lambda,r} & =  \int_D \, \left\langle
\nabla\left(\frac{f^2}{W}\right) \, , \, \nabla W \right\rangle \dmu_{\lambda,r} +
\int_{S(y,r)} \, \frac{\partial W}{\partial n} \, \frac{f^2}{W} \,
d\mu_{\lambda,r}^S \\
& =  2 \, \int_D \, \frac fW \, \langle \nabla f , \nabla W \rangle \dmu_{\lambda,r} \, - \,
\int_D \, \frac{f^2}{W^2} \, |\nabla W|^2 \dmu_{\lambda,r} + \int_{S(y,r)} \, \frac{\partial
W}{\partial n} \, \frac{f^2}{W} \,
d\mu_{\lambda,r}^S \\
& =  - \, \int_D  \, \left| \frac fW \, \nabla W - \nabla f\right|^2 \dmu_{\lambda,r} \, + \,
\int_D \, |\nabla f|^2 \dmu_{\lambda,r} + \int_{S(y,r)} \, \frac{\partial W}{\partial n} \,
\frac{f^2}{W} \, d\mu_{\lambda,r}^S  \, .
\end{align*}
\end{proof}

\subsection{Localizing around the obstacle.\\\\}\label{subpres}

In the sequel $U$ will be an open ball centered at $y$. Without loss of generality (if
necessary) we may assume that $y=(a,0)$ for some $a\in \R^+$, $0$ being the null vector of
$\mathbb R^{d-1}$. The (non normalized) normal vector field at the boundary of $B(y,r)$, pointing inward $D$, 
is thus $x-y = (x^1-a,\bar x)\in \R \times \R^{d-1}$. We shall denote by $n$ the normalized
inward normal vector field.

We shall see how to use the two previous lemma in our context. Indeed let $h>0$ and
assume that one can find a Lyapunov function $W$ such that $LW \leq - \theta W$ for $|x-y|\leq
r+2h$ and $\partial W/\partial n \leq 0$ on $|x-y|=r$. Choose some smooth function $\psi$ such
that $\BBone_{\{|x-y|\leq r+2h\}} \geq \psi \, \geq \BBone_{\{|x-y|\leq r+h\}}$ and, for some
$\varepsilon >0$,
$$\parallel \nabla \psi\parallel_\infty \leq (1+\varepsilon)/h \, .$$ Applying Lemma
\ref{lemIPP}  to $\psi \, f$  we obtain thanks to \eqref{eqbordnul}
\begin{eqnarray*}
\int_{r<|x-y|<r+h} \, f^2 \, d\mu_{\lambda,r} & \leq & \int_D \, (\psi \, f)^2 \,
d\mu_{\lambda,r} \\ & \leq & \frac{1}{\theta} \, \int_D \, \frac{-LW}{W} \, (\psi \, f)^2
\, d\mu_{\lambda,r} \\ & \leq & \frac {1}{\theta} \, \int_{r<|x-y|<r+2h}\, |\nabla f|^2 \,
d\mu_{\lambda,r} \, \\ & \, & \,  + \, \frac{1}{\theta} \,
\left(\frac{1+\varepsilon}{h}\right)^2 \, \int_{r+h<|x-y|<r+2h} \, f^2 \, d\mu_{\lambda,r} \,
.
\end{eqnarray*}
Me may first let $\varepsilon$ go to $0$. Next choose $U=B(y,r+h)$, $\BBone_{\{|x-y|> r\}}
\geq \chi \, \geq \BBone_{\{|x-y|\geq r+h\}}$. Using a similar argument as before we may
assume that actually $\parallel \nabla \chi\parallel_\infty = \frac 1h$.

The previous inequality applied to $f-m$ yields
\begin{equation}\label{eqhowtoget}
\int_U \, (f-m)^2 \, d\mu_{\lambda,r} \leq \frac {1}{\theta} \, \int_D \, |\nabla f|^2 \,
d\mu_{\lambda,r} +  \frac{1}{\theta \, h^2} \, \int_{\mathbb R^d} \chi \, (f-m)^2 \,
d\mu_{\lambda,r}
\end{equation}
i.e. \eqref{eqpoidsy} is satisfied with 
\begin{equation}\label{eqhowtoget2}
C = \frac 1\theta \quad \textrm{ and } \quad p = \lambda \, \theta \, h^4 \, ,
\end{equation}
provided the latter is larger than 1. We may thus apply lemma \ref{lempoidsy} and obtain
\begin{lemma}\label{lemlypoids}
Let $h>0$. Assume that one can find a Lyapunov function $W$ such that $LW \leq - \theta W$ for $|x-y|\leq
r+2h$ and $\partial W/\partial n \leq 0$ on $|x-y|=r$.

Then, provided $\lambda \, \theta \, h^4 >1$, $$ \Var_{\mu_{\lambda,r}}(f) \leq \frac{h^2 \, (2+ (\theta +\lambda) \, h^2)}{\lambda \, \theta \, h^4 \, - \, 1} \, \int_D \, |\nabla f|^2 \,
d\mu_{\lambda,r} \, .$$
\end{lemma}
\medskip

Hence all we have to do is to find a ``good'' Lyapunov
function.

We shall exhibit some Lyapunov
function $W_y$ near the obstacle. Recall that we may assume that $y=(a,0)$ for some non
negative real number $a$ and write $x-y = (x^1-a,\bar x)$. For $|\bar x|\leq r+2h$ define
$$W_y(x^1,\bar x) = (r+2h+\varepsilon)^2 - |\bar x|^2 \, .$$ Then $\nabla W_y(x^1,\bar x)=(0,
- 2\bar x)$ and
\begin{equation}\label{eqbordnul}
\frac{\partial W_y}{\partial n}(x^1,\bar x) = - \, \frac{2 |\bar x|^2}{|x-y|} \, \leq 0 \, .
\end{equation}
 Now $LW_y = - (d-1) \, + \, 2 \, \lambda \, |\bar x|^2$ so that $LW_y
\leq -  \,  2 \, \lambda \, W_y$ provided
\begin{equation}\label{eqcondition0}
d-1 \, \geq  \, 2 \, \lambda \,  (r+2h+\varepsilon)^2 \, .
\end{equation}
As before we may let $\varepsilon$ go to $0$ so that we obtain \eqref{eqhowtoget} with
$\theta=2 \lambda$ and $p=2 \lambda^2 \, h^4 >1$.

Choosing $h = b/\sqrt \lambda$, with $p=2b^4 >1$, we see that we must have $d \geq 7$ and $r \sqrt \lambda \leq
\sqrt{(d-1)/2} - 2b$. Finally we have shown

\begin{proposition}\label{propunobys3}
If $p=2b^4 >1$  and $r \sqrt \lambda \leq \sqrt{(d-1)/2} \, - \, 2b$, so that $d \geq 7$. Then
 the measure
$\mu_{\lambda,r}$ satisfies a Poincar\'e inequality \eqref{eqpoinc} with $$C_P(\lambda,y,r) \leq \, \frac{1}{\lambda} \, \frac{b^2(3b^2+2)}{2b^4 -1} \, .$$
\end{proposition}
\smallskip

Notice that for small enough $r$ and large dimension, this result is better than all we obtained previously, except for $y=0$ where we recover asymptotically ($d$ hence $b$ growing to infinity) the upper bound of Proposition \ref{propunobcentre}. What is interesting here is that the result is also true for all $y$'s.
The dimension dependence clearly indicates that, even for small $r$'s, we presumably did not find the
good Lyapunov function. 

Also notice that if we define $\beta = \frac{2b \sqrt 2}{\sqrt {d-1}}$ the condition on $r$ read
\begin{equation}\label{eqsizer}
r \sqrt \lambda \leq (1-\beta) \, \sqrt{(d-1)/2} \quad \textrm{for some $\beta$ such that} \quad 1>\beta>\frac{2^{5/4}}{\sqrt {d-1}} \, .
\end{equation}

In the next three subsections we shall adapt the previous method in order to cover all dimensions but for a far enough obstacle.  
\medskip

\subsection{Localizing away from the obstacle and the origin.}\label{subloin} \quad 
Consider now $W(x)=|x|^2$ so that for 
$1>\eta>0$,
$$LW(x)= d - 2 \lambda \, W(x) \leq - \, 2 \lambda \, (1-\eta) \, W(x) \quad \textrm{ for }
|x|\geq \sqrt{\frac{d}{2 \lambda \eta}} \, .$$ We assume that $a-r-3h \geq  \,
\sqrt{\frac{d}{2 \lambda \eta}}$, in particular $a$ is large enough. Let $g$ be a smooth
function compactly supported in $|x|\geq \sqrt{\frac{d}{2 \lambda \eta}}$. For all $1\leq
\varepsilon \leq 2$ we apply lemma \ref{lemIPP} in $|x-y|\geq r+\varepsilon h$ with $g$,
i.e.
$$
\int_{|x-y|\geq r+\varepsilon h} \, \frac{-LW}{W} \, g^2 \, d\mu_{\lambda,r} \, \leq \,
\frac 12 \,  \int_{|x-y|\geq r+\varepsilon h} \,  |\nabla g|^2 \, d\mu_{\lambda,r} \, + \,
\frac 12 \, \int_{|x-y|= r+\varepsilon h} \, \frac{\partial W}{\partial n} \, \frac{g^2}{W}
\, d\mu_{\lambda,r}^\varepsilon \, ,$$ where $\mu_{\lambda,r}^\varepsilon$ denotes the trace of $\mu_{\lambda,r}$ on the sphere $|x-y|=r+\varepsilon \, h$.

It yields for all $\varepsilon$ as before
\begin{eqnarray*}
\int_{|x-y|\geq r+ 2h} \, g^2 \, d\mu_{\lambda,r} & \leq & \int_{|x-y|\geq r+ \varepsilon
h} \, g^2 \, d\mu_{\lambda,r} \\ & \leq & \frac{1}{2 \lambda \, (1-\eta)} \, \int_{|x-y|\geq
r+\varepsilon h} \, \frac{-LW}{W} \, g^2 \, d\mu_{\lambda,r} \\ & \leq & \frac{1}{4 \lambda
\, (1-\eta)} \, \int_{|x-y|\geq r+ h} \,  |\nabla g|^2 \, d\mu_{\lambda,r} \, + \, \\ & &
\, + \, \frac{1}{4 \lambda \, (1-\eta)} \, \int_{|x-y|= r+\varepsilon h} \, \frac{\partial
W}{\partial n} \, \frac{g^2}{W} \, d\mu_{\lambda,r}^\varepsilon \, .
\end{eqnarray*}
Remark that $(1/W) \, |\frac{\partial W}{\partial n}| \leq 2/|x|$ so that we obtain
$$\int_{|x-y|\geq r+ 2h} \, g^2 \, d\mu_{\lambda,r} \leq $$ $$ \qquad \leq \, \frac{1}{4 \lambda
\, (1-\eta)} \, \left(\int_{|x-y|\geq r+ h} \,  |\nabla g|^2 \, d\mu_{\lambda,r} +
\frac{2}{a-r-2h} \, \int_{|x-y|= r+\varepsilon h} g^2 \,
d\mu_{\lambda,r}^\varepsilon\right) \, .$$ We integrate the previous inequality with respect
to $\varepsilon$ for $1\leq \varepsilon \leq 2$. It follows
$$\int_{|x-y|\geq r+ 2h} \, g^2 \, d\mu_{\lambda,r} \leq $$ $$ \quad \leq \, \frac{1}{4 \lambda
\, (1-\eta)} \, \left(\int_{|x-y|\geq r+ h} \,  |\nabla g|^2 \, d\mu_{\lambda,r} +
\frac{2}{a-r-2h} \, \int_1^2 \, \int_{|x-y|= r+\varepsilon h} g^2 \,
d\mu_{\lambda,r}^\varepsilon \, d\varepsilon\right) \, .$$ We shall study the second term on
the right hand side, writing $x=y+(r+\varepsilon h)u$ for $u\in S^{d-1}$, so that if
$\sigma(du)$ denotes the (non normalized) surface measure on $S^{d-1}$ it holds
\begin{eqnarray*}
\int_1^2 \, \int_{|x-y|= r+\varepsilon h} g^2 \, d\mu_{\lambda,r}^\varepsilon \,
d\varepsilon & = & \int_1^2 \, \int_{S^{d-1}} g^2(y+(r+\varepsilon h)u) \, Z^{-1}_{\lambda,r} \, e^{-\lambda
|y+(r+\varepsilon h)u|^2} \, \\ & & \qquad \qquad \qquad \qquad \left(r+\varepsilon
h\right)^{d-1} \, \sigma(du) d\varepsilon\\ & = & \frac{1}{h} \, \int_{r+h\leq |x-y|\leq
r+2h} g^2 \, d\mu_{\lambda,r} \, .
\end{eqnarray*}
We have thus obtained
\begin{lemma}\label{lemloindey}
Assume that for some $0<\eta<1$ and $h>0$, $a-r-3h \geq  \, \sqrt{\frac{d}{2 \lambda
\eta}}$. If $g$ is a smooth function compactly supported in $|x|\geq \sqrt{\frac{d}{2 \lambda
\eta}}$, then $$\int_{|x-y|\geq r+ 2h} \, g^2 \, d\mu_{\lambda,r} \leq \, $$ $$\qquad \leq
\frac{1}{4 \lambda \, (1-\eta)} \, \left(\int_{|x-y|\geq r+ h} \, |\nabla g|^2 \,
d\mu_{\lambda,r} + \frac{2}{h \, (a-r-2h)} \, \int_{r+h\leq |x-y|\leq r+2h} g^2 \,
d\mu_{\lambda,r}\right) \, .$$
\end{lemma}
\medskip

\subsection{Localizing away from the origin for a far enough obstacle.} \quad Now we shall put together the previous two localization procedures. 

Remark that, during the proof of lemma \ref{lemlypoids} (more precisely with an immediate modification), we have shown the following : provided we can find a Lyapunov function in the neighborhood $|x-y|\leq 3h$ of the obstacle
$$\int_{r<|x-y|<r+2h} \, f^2 \, d\mu_{\lambda,r} \leq  \frac {1}{\theta} \, \int_{r<|x-y|<r+3h}\, |\nabla f|^2 \,
d\mu_{\lambda,r}  + \, \frac{1}{\theta \, h^2} \,
 \, \int_{r+2h<|x-y|<r+3h} \, f^2 \, d\mu_{\lambda,r} \,
,$$ so that using the Lyapunov function $W_y$ in subsection \ref{subpres} (yielding $\theta = 2 \, \lambda$) we have, provided $d-1 \geq 2 \lambda \, (r+3h)^2$,
\begin{equation}\label{eqjesaisplus}
\int_{r<|x-y|<r+2h} \, f^2 \, d\mu_{\lambda,r} \leq  \frac {1}{2 \lambda} \, \int_{r<|x-y|<r+3h}\, |\nabla f|^2 \,
d\mu_{\lambda,r}  + \, \frac{1}{2 \lambda \, h^2} \,
 \, \int_{r+2h<|x-y|<r+3h} \, f^2 \, d\mu_{\lambda,r} \, 
.
\end{equation}
\medskip

Assume in addition that the conditions in lemma \ref{lemloindey} are fulfilled. For a smooth function $g$ with compact support included in $|x|\geq \sqrt d$, we denote
$$A=\int_{r<|x-y|<r+2h} \, g^2 \, d\mu_{\lambda,r} \, ,$$ $$B=\int_{|x-y|\geq r+ 2h} \, g^2 \,
d\mu_{\lambda,r} \, ,$$ and $$C= \int_{|x-y|\geq r} \, |\nabla g|^2 \, d\mu_{\lambda,r} \, .$$
According to what precedes, we obtain
$$
A  \leq  \frac {1}{2 \, \lambda} \, \left(C  + \, \frac{1}{h^2} \, B\right) \quad \textrm{ and } \quad B \leq
\frac{1}{4 \lambda \, (1-\eta)} \, \left(C + \frac{2}{h(a-r-2h)} \, A\right) \, .
$$

For simplicity we choose arbitrarily $\eta =1/2$. Hence, provided $$2 \, \lambda^2 \, h^3 \, (a-r-2h) \, > \, 1 \, ,$$ we obtain $$A \, \leq \, \frac{1}{2 \, \lambda} \, \left(1 + \frac{1}{2 \, \lambda \, h^2}\right) \, \left(1 - \, \frac{1}{2 \, \lambda^2 \, h^3 \, (a-r-2h)}\right)^{-1} \, C \, ,$$ i.e. $$A \, \leq \, \frac{h}{2} \, \left(1 + 2 \, \lambda \, h^2\right) \, \left(\frac{a-r-2h}{2 \, \lambda^2 \, h^3 \, (a-r-2h) - 1}\right) \, C \, ,$$ and $$B \, \leq \, \frac{1}{2 \lambda} \, \left( 1 + \frac{1 + 2 \, \lambda \, h^2}{2 \, \lambda^2 \, h^3 \, (a-r-2h) - 1}\right) \, C \, .$$

To obtain tractable bounds for $A+B$, we shall first assume that $\lambda=1$ and then use the homogeneity property.
\smallskip

Now we will choose $h=\frac b3 \, \sqrt{(d-1)/2}$ for some $b<1$ (recall that $r+3h \leq \sqrt{(d-1)/2}$). So we must have $r \leq (1-b) \, \sqrt{(d-1)/2}$ and $a \geq \sqrt d + \sqrt{(d-1)/2}$. 

Finally $2h^3(a-r-2h)>1$ as soon as $a \geq \frac{27 \, \sqrt 2}{b^3 (d-1)^{3/2}} + \sqrt{(d-1)/2}$. 

In order to get more tractable constants we will choose $2h^3(a-r-2h)-1 \geq 1/2$ hence $a \geq \frac{81}{b^3 \sqrt 2 \, (d-1)^{3/2}} + \sqrt{(d-1)/2}$, so that $$B \leq \left(\frac 32 + \frac{b^2 (d-1)}{9}\right) \, C \, .$$ The function $u \mapsto u/(2h^3u -1)$ being non increasing we similarly get $$ A \leq \frac 32 \, \left(1 + \frac{9}{b^2 (d-1)}\right) \, C \, .$$

This yields
\begin{lemma}\label{lemjesaispus}
Let $0<b<1$. Assume that $\lambda=1$, $r\leq (1-b) \, \sqrt{(d-1)/2}$ and $$|y|>\sqrt d + \, \sqrt{(d-1)/2} +  \frac{81}{b^3 \sqrt 2 \, (d-1)^{3/2}} \, .$$ Then, for all smooth function $g$, compactly supported in $|x|\geq \sqrt d$, it holds $$\int \, g^2 \, d\mu_{\lambda,r} \, \leq \, K \, \int \,  |\nabla g|^2 \, d\mu_{\lambda,r} \, ,$$ with $$K = 3 +  \frac{b^2 (d-1)}{9} + \frac{27}{2b^2 (d-1)}\, .$$
\end{lemma}
\medskip

\subsection{Localizing around the origin for a far enough obstacle.} \quad It remains now to follow the method in \cite{BBCG,CGZ}. Let $f$ be a smooth function with compact support. Assume that we are in the situation of lemma \ref{lemjesaispus} (in particular $\lambda=1$). 

Recall that $\mu_{\lambda,r}$
restricted to the ball \{$|x|\leq 1+\sqrt d$\} is just the gaussian measure restricted to the ball (since this ball does not intersect the obstacle), hence satisfies a Poincar\'e inequality with a constant
less than $\frac{1}{2}$. If $$m=\int_{|x|\leq 1+\sqrt d} f \,
d\mu_{\lambda,r}/\mu_{\lambda,r}(|x|\leq 1+\sqrt d) \, ,$$ we have $$\Var_{\mu_{\lambda,r}}(f) \leq
\int_D \, (f-m)^2 d\mu_{\lambda,r}$$ so that it is enough to control the second moment of
$\bar f=f-m$.

We write $$\bar f=\chi \, \bar f + (1-\chi) \, \bar f= \chi \bar f + g$$ where $\chi$ is
1-Lipschitz and such that $\BBone_{|x|\leq \sqrt d} \leq \chi \leq \BBone_{|x|\leq 1+\sqrt d}$. $g$ is
thus compactly supported in $|x|\geq \sqrt d$ so that we may apply what
precedes. In particular
\begin{eqnarray*}
\int_D \,  {\bar f}^2 \, d\mu_{\lambda,r} & \leq & 2 \, \int_{|x|\leq 1+\sqrt d}  {\bar f}^2 \,
d\mu_{\lambda,r} + 2 \, \int_D g^2 \, d\mu_{\lambda,r} \\ & \leq & 
\int_{|x|\leq 1+\sqrt d} \, |\nabla f|^2 \, d\mu_{\lambda,r} + 2 K \, \int_D \, |\nabla g|^2 \,
d\mu_{\lambda,r} \\ & \leq &  \int_{|x|\leq 1+\sqrt d} \, |\nabla f|^2 \,
d\mu_{\lambda,r} + 4 K \, \int_{x\in D , |x|\geq \sqrt d} \, |\nabla f|^2 \, d\mu_{\lambda,r} + \\
& & \qquad + 4 K \, \int_{1+\sqrt d \geq |x|\geq \sqrt d} \, {\bar f}^2 \, d\mu_{\lambda,r}\\
& \leq &  (1 + 2 K) \, \int_{|x|\leq 1+\sqrt d} \, |\nabla f|^2 \,
d\mu_{\lambda,r} + 4 K \, \int_{x\in D, |x|\geq \sqrt d} \, |\nabla f|^2 \,
d\mu_{\lambda,r}\\ & \leq &  (1+ 6 K) \,
\int_D \, |\nabla f|^2 \, d\mu_{\lambda,r} \, .
\end{eqnarray*}
\medskip

We have thus proved, using \eqref{eqhomo}
\begin{proposition}\label{propcomprpetit}
 Assume that, for some $0<b<1$, we have $r \sqrt \lambda \leq (1-b) \, \sqrt{(d-1)/2}$ and $$|y| \sqrt \lambda > \sqrt d + \, \sqrt{(d-1)/2} +  \frac{81}{b^3 \sqrt 2 \, (d-1)^{3/2}}  \, .$$ Then
 the measure
$\mu_{\lambda,r}$ satisfies a Poincar\'e inequality \eqref{eqpoinc} with $$C_P(\lambda,y,r) \leq \, \frac{1}{\lambda} \, \left(1 + 6K\right) \, ,$$ $K$ being given in lemma \ref{lemjesaispus}.
\end{proposition}
\bigskip

\subsection{A general result for small radius.} \quad We can gather together all the previous results. For the sake of simplicity the next theorem is not optimal, but readable.

\begin{theorem}\label{thmsmallradius}
There exists some universal constant $c$ such that if $$r\sqrt \lambda \leq \frac 12 \, \sqrt{(d-1)/2} \, ,$$  the measure
$\mu_{\lambda,r}$ satisfies a Poincar\'e inequality \eqref{eqpoinc} with $$C_P(\lambda,y,r) \leq \, \frac{c}{\lambda} \, .$$
\end{theorem}
\begin{proof}
If $d$ is big enough ($d\geq 33$) we may use Proposition \ref{propunobys3}. If $d\leq 33$ and $|y| \sqrt \lambda$ large, we may apply Proposition \ref{propcomprpetit} with $b=1/(2 \sqrt{d-1}) $. Finally, if $d\leq 33$ and $|y| \sqrt \lambda$  is small we may use Proposition \ref{propunoby}.
\end{proof}

\begin{remark}
In comparison with Proposition \ref{propunobys3}, we have spent a rather formidable energy in order to cover the small dimension situation. But the alternate method we have developed for large $|y|$ will be useful in other contexts.

It is also worth noticing that we have used Proposition \ref{propunoby} that cannot be extended to more than one obstacle.
\end{remark}
\medskip

\subsection{Using curvature.\\\\}

Let us finish this section by an alternative approach based on usual curvature argument, which however presents some additional technical difficulties when boundaries of domain are not convex. Indeed, a renowned method to get functional inequalities (like Poincar\'e or logarithmic Sobolev inequalities) is to use curvature assumptions, for instance the Bakry-Emery criterion. This criterion extends to manifolds with convex boundary. In a series of papers (starting with \cite{Wang05,Wang07} and ending with \cite{Wang11}), Feng-Yu Wang developed a new method in order to cover some cases with non-convex boundary.  

Let us (briefly) see how these ideas may apply here. We may consider $D$ as a flat $d$-dimensional manifold with a (non-convex) boundary $\partial D = \{|x-y|=r\}$. The Ricci tensor on $D$ is thus the nul tensor, while the second fundamental form on $\partial D$ is $- (1/r) \, Id$. In \cite{Wang11}, Wang introduces a modified curvature tensor, reducing here to
\begin{equation}\label{eqcurv}
Ric_m = \left(\lambda \,  - \, \frac 12 \, \varphi^2 \, (L\varphi^{-2})\right) \, Id 
\end{equation} 
where $\varphi$ is a smooth function defined on $D$, satisfying
\begin{equation}\label{eqphi}
\varphi \geq 1 \quad , \quad \nabla \varphi \, \bot \, T_x\partial D \quad , \quad \partial_n \log(\varphi) \geq \frac 1r \, \textrm{ on } \partial D \, .
\end{equation}
As a byproduct of Corollary 1.2 in \cite{Wang11} (see remark and (1.3) therein), it is shown that, if 
\begin{equation}\label{eqcurve2}
K_\varphi = \inf_{x \in D} \, \left(\lambda \,  - \, \frac 12 \, \varphi^2 \, (L\varphi^{-2})\right) > 0
\end{equation}
then $$C_P(\lambda,\{y\},r) \, \leq \, \frac{2 \, \parallel \varphi\parallel_\infty^2} {K_\varphi} \, .$$
As for a Lyapunov function, it remains to find a ``good'' function $\varphi$. Following the arguments by Wang, we will choose $$\varphi(x) = e^{\frac 12 \, h(|x-y|^2)} \, , \textrm{ for some non-negative $h$ defined on $[r^2,+\infty[$} \, .$$ We thus have $$\nabla \varphi(x)= h'(r^2) \varphi(x) \, (x-y) \, \bot \, T_x\partial D$$ and $$\partial_n \log(\varphi)(x) = r \, h'(r^2) \textrm{ for $x\in \partial D$} \, ,$$ so that \eqref{eqphi} is satisfied as soon as 
\begin{equation}\label{eqderivhbord}
h'(r^2) \, \geq \frac{1}{r^2} \, .
\end{equation}
Next, if we define $u=|x-y|^2$, $$\lambda \,  - \, \frac 12 \, \varphi^2 \, (L\varphi^{-2})(x) = \lambda \, (1-u \, h'(u)-\langle y,x-y\rangle \, h'(u)) + \frac12 \, d \, h'(u) +u \, (h''(u) - (h'(u))^2) \, .$$ Hence the best lower bound we can get is $$\lambda \,  - \, \frac 12 \, \varphi^2 \, (L\varphi^{-2})(x) \geq \lambda \, (1-u \, h'(u)- |y| \, u^{\frac 12} \, |h'(u)|) + \frac12 \, d \, h'(u) +u^2 \, (h''(u) - (h'(u))^2) \, .$$ Recall that we are looking for a lower bound for the latter expression.

For $u=r^2$, it seems that the best possible choice is $h'(r^2)=1/r^2$, yielding $$- \lambda  \, \frac{|y|}{r} + \frac 12 \, \frac{d-2}{r^2} + r^2 \, h''(r^2) > 0 \, .$$ It also seems that we have to choose $h''\leq 0$ (otherwise the term $u^2 \, (h''(u) - (h'(u))^2)$ will have a tendency to become very negative for large values of $u$).


A reasonable choice seems to be $$h'(u) = \frac{1}{r^2} - \theta(u-r^2)^k \quad \textrm{ for } u-r^2 \leq \frac{1}{(\theta \, r^2)^\frac1k} \, ,$$ and then $h'=0$. $h$ is not twice differentiable at $u =r^2 + \frac{1}{(\theta \, r^2)^\frac1k}$ but this is not relevant.

For $k=1$ we thus have to assume, for $u=r^2$ $$- \lambda  \, \frac{|y|}{r} + \frac 12 \, \frac{d-2}{r^2} - \theta \, r^2 > 0 \, ,$$ and, for $u =r^2 + \frac{1}{\theta \, r^2}$, $$ \lambda - \theta \, \left(r^2 + \frac{1}{\theta \, r^2}\right) > 0 \, .$$ So that even for $y=0$, the previous method should only be useful (it remains to check what happens for $r^2\leq u \leq r^2 + \frac{1}{\theta \, r^2}$) for small $r$'s. To save place we do not include the similar discussion for $k\geq 2$ which yields similarly bad restrictions on the size of $r$.

In conclusion this approach is not easier than the one we have previously used and seems to provide worse bounds than the one obtained in the previous sections. 
\smallskip


\section{First lower bounds.}\label{subseclower}

Here we shall see how it is possible to derive some lower bound for $C_P(\lambda,y,r)$ when $y\neq 0$. 
\smallskip

To this end, first recall the link between the Poincar\'e constant and the exponential moments of hitting times shown in \cite{CGZ} Proposition 3.1

\begin{proposition}\label{prophit}
Let $U$ be a subset of $D$ and $T_U$ denotes the hitting time of $U$. Then, for all $x\in D$, $$\E_x\left(e^{\theta \, T_U}\right)<+\infty \quad \textrm{ for all $\theta<\theta(U)$, with } \theta(U)=\frac{\mu_{\lambda,r}(U)}{32 \, C_P(\lambda,y,r)} \, \, .$$
\end{proposition}
Actually, \cite{CGZ} only dealt with diffusion processes, without reflection. But the proof of this Proposition lies on three facts which are still true here: the symmetry of $\mu_{\lambda,r}$, the existence of a density for the law at time $t>0$ of the process starting at any $x$, the results of Proposition 1.4 and Remark 1.6 in \cite{CG-FA} which hold true for general Markov processes with a square gradient operator.

According to Proposition \ref{prophit}, if for some $x$,
\begin{equation}\label{eqcomp}
\E_{x}\left(e^{\theta \, T_U}\right) = + \infty  \quad \textrm{ then} \quad C_P(\lambda,y,r) \geq \frac{\mu_{\lambda,r}(U)}{32 \, \theta} \, .
\end{equation}
\smallskip

So, for $y=(a,0)$ with $a \geq 0$ as before, consider the half space $U(r)=\{x \, ; \, x^1 \geq a-r-\varepsilon(r) \, \}$ where $\varepsilon(r)>0$ is chosen in such a way that $\mu_{\lambda,r}(U(r))\geq 1/2$. Of course, such a $\varepsilon(r)$ always exists.

The starting point $x^0$ is chosen as $x^0=(b,0)$ with $b<a-r-\varepsilon(r)$, so that, up to time $T_{U(r)}$ the process $X_.$ is simply the ordinary Ornstein-Uhlenbeck (since the process did not meet the obstacle). So, $T_{U(r)}$ is the hitting time of $a-r-\varepsilon(r)$ starting from $b$ for a linear O-U process.

Now assume that $\lambda=1$ and $r<\frac{\sqrt{d-1}}{2 \sqrt 2}$, as in Theorem \ref{thmsmallradius}. We may choose  
$j_d < - \, \frac{\sqrt{d-1}}{2 \sqrt 2}$ such that for all $a\geq 0$ and all $r<\frac{\sqrt{d-1}}{2 \sqrt 2}$, $\mu_{1,r}(x^1>j_d) \geq \frac 12$. Starting with $b<j_d$, denote by $T_d$ the hitting time of $V=\{x \, , \, x^1>j_d\}$. This time is the same for all $a$ and $r$ as before, since the corresponding processes coincide up to $T_d$. In particular, it is known that there exists some $\theta_d$ such that $\E_{x^0}(e^{\theta_d \, T_d}) = + \infty$ (see e.g. \cite{CGZ} Proposition 5.1 for a proof). Hence, using the homogeneity property \eqref{eqhomo} we have shown that
\begin{proposition}\label{propmin}
In the situation of Theorem \ref{thmsmallradius}, there exist constants $c_d>0$ (depending on the dimension $d$) such that $$C_P(\lambda,y,r) \geq \frac{c_d}{\lambda} \, .$$
\end{proposition}

Of course, what is expected is a dimension free lower bound, as in the case $y=0$. The main (only) advantage of the previous proposition, is that it allows us to only consider large enough dimensions.
\smallskip

If $a-r>0$ (i.e the origin belongs to $D$), we may choose any $x^0=(-b,0)$ with $b>1$ and look at the hitting time of $V=\{x^1\geq -1\}$. According to the previous discussion we thus have to look at the exponential moments of the hitting time $T$ of $-1$ for a linear O-U process (with $\lambda=1$). Notice that the strip $-1\leq x^1 \leq 0$ has a $\mu_{1,r}$ measure larger than the $\gamma_1$ measure of the same strip, which is positive and dimension free. The previous argument thus shows that
\begin{equation}\label{eqminapos}
\textrm{If $|y|-r>0$, then } \quad C_P(\lambda,y,r) \geq \frac c\lambda \quad \textrm{for some universal constant $c$.}
\end{equation}

Next, for $d$ large enough, the gaussian measure $\gamma_1$ concentrates in a small shell $\{\frac d2 (1- \varepsilon(d)) \leq |x|^2 \leq \frac d2 ( 1+\varepsilon(d)) \}$, with $\varepsilon(d)$ going to $0$ as $d\to +\infty$, so that, if $r<\frac{\sqrt{d-1}}{2 \sqrt 2}$, the $\gamma_1$ measure of the obstacle will become small (say less than $\frac 14$), whatever the position of $y$.

Recall that the law of $T=T_{U(r)}$ is exactly known only in the case $a-r-\varepsilon(r)=¿$ (\cite{PY}) and is given by the density $$p(t) = \frac{b}{\sqrt{2\pi}} \, (\sinh(t))^{-\frac 32} \, e^{\left(- \, \frac{b^2 \, e^{-t}}{2 \, \sinh(t)} + \, \frac t2\right)} \, \BBone_{t>0} \, ,$$ so that $\E_{-b}(e^{\theta T})=+\infty$, as soon as $\theta > 1$. 
\smallskip

If $a-r\leq 0$ we choose for $V$ the ball of radius $\sqrt d$ centered at the origin and intersected with $D$. A similar reasoning shows that  $\mu_{1,r}(V) \geq \frac 14$. It remains to look at the hitting time $T$ of $V$ starting from the outside. But this latter point seems to be non obvious and we did not completely fix this problem.
\medskip


\section{Replacing balls by squares or hypercubes.}\label{secunoby3}

When $\lambda \to \infty$ we were only able to get nice estimates provided the size of the obstacle goes to $0$ sufficiently quickly, while Proposition \ref{propunoby} furnishes an exploding upper bound for the Poincar\'e constant. 

The case $y=0$ is somewhat special since the gaussian measure $\mu_{\lambda,r}$ weakly converges to the uniform measure on the sphere and we recovered up to the constants the correct Poincar\'e constant. 

If $y\neq 0$ the situation seems to become better since, as $\lambda \to \infty$, $\mu_{\lambda,r}$ weakly converges to a Dirac mass $\delta_m$ where $m=0$ if $0 \in D$, and $m\in \partial D$ is the closest point to the origin if the origin does not belong to $D$. One should think that the Poincar\'e constant thus converges to $0$ too. This is not the case and because it is simpler we shall see why, first  changing the geometry of the obstacle, replacing balls by hypercubes. We will see that a squared obstacle becomes a long time trap, and the Poincar\'e constant explodes. We shall develop this point below.

\subsection{A neighboring example in dimension 2, with a squared obstacle.}\label{subsecsquare} \quad For simplicity we assume that $d=2$ and replace the disc $\{|x-y|<r\}$ with $y=(a,0)$ ($a>0$) by a square, $\{|x^1-a|<r \, , \, |x^2|<r\}$. To avoid complications we may ``smooth the corners'' for the boundary to be smooth (replacing $r$ by $r+\varepsilon$), so that existence, uniqueness and properties of the reflected process are similar to those we have mentioned for the disk. We will also assume that $a-r>0$, so that $\mu_{\lambda,r}$ weakly converges to the Dirac mass at the origin as $\lambda \to +\infty$.
\smallskip

Consider the process $X_t$ starting from $x=(a+r,0)$. Denote by $S(r)$ the exit time of $[-r,r]$ by the second coordinate $X_.^2$. Up to time $S(r)$, $X_.^2$ is an Ornstein-Uhlenbeck process, starting at $0$, $X_.^1$ is an Ornstein-Uhlenbeck process reflected on $a+r$, starting at $a+r$; and both are independent. Of course $S(r)=T_{U^c(r)}$ where $U(r)$ is the set $U(r)=\{x^1\geq a+r \, ; \, |x^2|\leq r\}$, and by symmetry $\mu_{\lambda,r}(U^c(r))\geq \frac 12$. 

According to Proposition \ref{prophit}, if 
\begin{equation}\label{eqcomp2}
\E_{(a+r,0)}\left(e^{\theta \, S(r)}\right) = + \infty  \quad \textrm{ then} \quad C_P(\lambda,y,r) \geq \frac{1}{32 \, \theta} \, .
\end{equation}
 The exponential estimate we are looking for is just the exponential moment of the exit time of a symmetric interval for a linear Ornstein Uhlenbeck process starting from the origin. It is thus clear that, for a fixed $r$, the Poincar\'e constant will go to infinity as $\lambda \to +\infty$.

We will try to better understand this behavior, in particular to get quantitative bounds.
\smallskip

For the linear Brownian motion it is well known, (see \cite{RY} Exercise 3.10) that $$E_0\left(e^{\theta \, S(r)}\right)= \frac{1}{\cos(r \, \sqrt{2\theta})}<+\infty$$ if and only if $$\theta \leq \frac{\pi^2}{8 \, r^2} \, .$$ Surprisingly enough (at least for us) a precise description of the Laplace transform of $S(r)$ for the O-U process is very recent: it was first obtained in \cite{GJY}. A simpler proof is contained in \cite{GJ} Theorem 3.1. The result reads as follows
\begin{theorem}\label{thmhit}\{\textbf{see \cite{GJY,GJ}}\}
If $S(r)$ denotes the exit time from $[-r,r]$ of a linear O-U process, then for $\theta\geq 0$, $$E_0\left(e^{- \, \theta \, S(r)}\right) = \frac{1}{_1F_1\left(\frac{\theta}{2\lambda}\, , \, \frac 12 \, , \, \lambda \, r^2\right)} \, ,$$ where $_1F_1$ denotes the confluent hypergeometric function.
\end{theorem}
The function $_1F_1$ is also denoted by $\Phi$ (in \cite{GJY} for instance) or by $M$ in \cite{hand} (where it is called Kummer function) and is defined by
\begin{equation}\label{eqhyper}
_1F_1(a,b,z) = \sum_{k=0}^{+\infty} \, \frac{(a)_k}{(b)_k} \, \frac{z^k}{k!} \quad \textrm{ where } \quad (a)_k = a(a+1)...(a+k-1) \, , \, (a)_0=1 \, .
\end{equation}
In our case, $b=\frac 12$, so that $_1F_1$ is an analytic function, as a function of both $z$ and $\theta$. It follows that $\theta \mapsto E_0\left(e^{- \, \theta \, S(r)}\right)$ can be extended, by analytic continuation, to $\theta<0$ as long as $\lambda r^2$ is not a zero of $_1F_1(\frac{\theta}{2\lambda} \, , \, \frac 12 \, , \, .)$.
\smallskip

The zeros of the confluent hypergeometric function are difficult to study. Here we are looking for the first negative real zero. For $-1<a<0$, $b>0$, it is known (and easy to see) that there exists only one such zero, denoted here by $u$. Indeed $_1F_1(a,b,0)=1$ and all terms in the expansion \eqref{eqhyper} are negative for $z>0$ except the first one, implying that the function is decaying to $-\infty$ as $z \to +\infty$. However, an exact or an approximate expression for $u$ are unknown (see the partial results of Slater in \cite{slater,hand}, or in \cite{gatteschi}). Our situation however is simpler than the general one, and we shall obtain a rough but sufficient bound. 
\smallskip

First, comparing with the Brownian motion, we know that for all $\lambda>0$ we must have $$\frac{- \theta}{\lambda} \leq \frac{\pi^2}{8 (r\sqrt \lambda)^2} \, .$$ So, if $\lambda \, r^2 > \pi^2/8$ and $- \theta/2 \lambda \geq 1/2$, the Laplace transform (or the exponential moment) is infinite. We may thus assume that $- \theta/2 \lambda < 1/2$.

Hence, for $_1F_1\left(\frac{\theta}{2\lambda}\, , \, \frac 12 \, , \, \lambda \, r^2\right)$ to be negative it is enough that 
\begin{eqnarray*}
1 &<& \frac{-\theta}{\lambda} \, \left((\lambda \, r^2) + \sum_{k=2}^{+\infty} \, \frac{(1+\frac{\theta}{2 \lambda})(2+\frac{\theta}{2 \lambda}) ...(k-1+\frac{\theta}{2 \lambda})}{(1+\frac 12)(2+\frac 12) ... (k-1+\frac 12)}  \, \frac{(\lambda \, r^2)^k}{k!}\right) \\ &<& \frac{-\theta}{ \lambda} \, \, \left(\sum_{k=1}^{+\infty}   \, \frac{(\lambda \, r^2)^k}{k!}\right) \, ,
\end{eqnarray*}
i.e.  
\begin{equation}\label{eqoumax}
\textrm{ as soon as } \quad \beta = - \theta \,  > \, \frac{\lambda}{e^{\lambda \, r^2} -1} \quad \textrm{ then } \quad 
\E_{(a+r,0)}\left(e^{\beta \, S(r)}\right) = + \infty \, ,
\end{equation}
 so that
\begin{equation}\label{eqpoincsquare}
\textrm{ if $\lambda \, r^2 > \pi^2/8$, \quad  then } \quad C_P(\lambda,y,r) \geq \frac{1}{32 \, \lambda} \, \left((e^{\lambda \, r^2} -1) \, \vee \, 1\right) = \, \frac{1}{32 \, \lambda} \, (e^{\lambda \, r^2} -1) \, .
 \end{equation}
 
This result is a good hint, but cannot be directly transposed to the case of a round obstacle, as we shall see later. Even for convex obstacles, the geometry of the boundary is particularly important.

\subsection{Small squared obstacle in dimension $d=2$.}\label{subsecsquare2} \quad We have just seen that the Poincar\'e constant quickly explodes as $r \sqrt\lambda \to +\infty$. It is natural to ask about its behavior  for small obstacles.

Let us come back to subsection \ref{subpres}. The function $W_y$ introduced therein is still a ``good'' Lyapunov function. Indeed, the only change in the squared situation is the normal derivative which becomes $x^1$ or $x^2=\bar x$ depending on the edge of the square, so that $$\frac{\partial W_y}{\partial n} \leq 0$$ on each edge. To be completely rigorous, we first have to smooth the corners of the square, but the reader will easily check that the previous property is still satisfied. Here we shall consider the square with edges of length $r/\sqrt 2$ which is thus included into the disc of radius $r$.

We may thus apply lemma \ref{lemIPP}, replacing the circle $\{|x-y|=r\}$ by the square (first after smoothing the corners, and then by taking limits) but still taking $U=B(y,r+h)$. We then obtain the analogue of what precedes \eqref{eqhowtoget} i.e., 
\begin{equation}\label{eqlyapsquare}
\int_{\{x\in D \, , \, |x-y|<r+h\}} \, f^2 \, d\mu_{\lambda,r} \leq \frac{1}{2\lambda} \, \int_D \, |\nabla f|^2 \, d\mu_{\lambda,r} + \frac{1}{2\lambda h^2} \, \int_{r+h<|x-y|<r+2h} \, f^2 \, d\mu_{\lambda,r} \, .
\end{equation}
But now, we can use what we know about the Poincar\'e constant in $D_h=\R^2 - B(y,r+h)$. Hence, if $m=\int_{D_h} f d\mu_{\lambda,r}$,
\begin{eqnarray*}
\Var_{\mu_{\lambda,r}}(f) &\leq& \int_{\{x\in D \, , \, |x-y|<r+h\}} \, (f-m)^2 \, d\mu_{\lambda,r} + \int_{D_h} \, (f-m)^2 \, d\mu_{\lambda,r} \\ &\leq& \frac{1}{2\lambda} \, \int_D \, |\nabla f|^2 \, d\mu_{\lambda,r} + \left(1+\frac{1}{2\lambda h^2}\right) \, \int_{D_h} \, (f-m)^2 \, d\mu_{\lambda,r} \\ &\leq& \left(\frac{1}{2\lambda} + \left(1+\frac{1}{2\lambda h^2}\right) C^b_P(\lambda,y,r+h)\right) \, \int_D \, |\nabla f|^2 \, d\mu_{\lambda,r} \, ,
\end{eqnarray*}
since $D_h\subseteq D$, and where $C^b_P(\lambda,y,r+h)$ denotes the Poincar\'e constant in $D_h$. 

For all this we need $r\sqrt \lambda \leq 1/(2 \sqrt2)$ and $h\sqrt \lambda \leq 1/(4\sqrt 2)$.
\medskip

\subsection{General results for a squared obstacle.}\label{subsecsquaregene} \quad 
We can now gather our results (just taking care of the definition of $r$ which is not the same in both previous subsections)

\begin{theorem}\label{thmsquare2}
In dimension $d=2$ let $D=\R^2 - S_r$ where $S_r=\{|x^1-a|<r \, , \, |x^2|<r\}$. Then there exists a constant $c$ such that 

if  $r\sqrt \lambda \leq \frac 18$,  then the Poincar\'e constant in $D$ satisfies  $C_P(\lambda,y,r) \leq \frac{c}{\lambda}$, while 

if $r\sqrt \lambda > \frac{\pi}{2\sqrt 2}$ then the Poincar\'e constant in $D$ satisfies $C_P(\lambda,y,r) \geq \frac{1}{32 \lambda} \, (e^{\lambda r^2}-1 )$ .
\end{theorem}

This result generalizes in any dimension, replacing the square by an hypercube (the $l^\infty$ ball). For the lower bound, we just have to look at the infimum of the exit times for $d-1$ independent O-U processes, for the upper bound to generalize what we have done before. The dimension dependence is then clear. For the upper bound, we have to include the hypercube in an euclidean ball, i.e. replace $r$ by $r\sqrt d$. But this dependence is exactly (up to universal constants) the one we have described in theorem \ref{thmsmallradius}. For the lower bound we have to include the dimension dependence in the explosion rate. The following can thus be easily proved, 

\begin{theorem}\label{thmsquare3}
Let $D=\R^d - S_r$ where $S_r=\{|x^1-a|<r \, , \, |\bar x|<r\}$. Then there exist universal constants $c<c', C, C'$ such that 

if  $r\sqrt \lambda \leq c$,  then the Poincar\'e constant in $D$ satisfies  $C_P(\lambda,y,r) \leq \frac{C}{\lambda}$, while 

if $r\sqrt \lambda > c'$ then the Poincar\'e constant in $D$ satisfies $C_P(\lambda,y,r) \geq \frac{C' \, e^{(\lambda r^2)}}{d \lambda}$.
\end{theorem}

These results are particularly interesting as they show a phase transition phenomenon: if the obstacle is small enough the Poincar\'e constant of the ordinary gaussian measure is preserved (up to universal constants), while it is drastically changed if the obstacle is not small enough, transforming the obstacle in a trap, even if it is located far from the origin, so that the restricted measure is very close to the ordinary gaussian measure. For small obstacle we may mimic the discussion we have previously done for obtaining some lower bound.

For a big obstacle, it should be interesting to get some upper bound too. In particular can we obtain an upper bound similar to the lower one, and which does not depend on the location of the obstacle ? Or on the contrary can we find a lower bound that depends on the location of the obstacle ?


\section{An isoperimetric approach.}\label{subsecisopsquare} \quad In this section, we present another approach for getting lower bounds. The easiest way to build functions allowing to see the lower bounds we have obtained in the previous subsection, is first to look at indicator of sets, hence isoperimetric bounds. 

We define the Cheeger constant $C_C(\lambda,y,r)$ as the smallest constant such that for all subset $A \subset D$ with $\mu_{\lambda,r}(A)\leq \frac 12$, 
\begin{equation}\label{eqcheeger}
C_C(\lambda,y,r) \, \mu_{\lambda,r}^S(\partial A) \geq  \, \mu_{\lambda,r}(A) \, .
\end{equation}
Recall that $\mu_{\lambda,r}^S(\partial A)$ denotes the surface measure of the boundary of $A$ in $D$ defined as $$\liminf_{h \to 0} \, \frac 1h \, \mu_{\lambda,r}(A_h/A)$$ where $A_h$ denotes the euclidean enlargement of $A$ of size $h$. The important fact here is that $A$ is considered as a subset of $D$. In particular, if we denote by $\partial S_r$ the boundary of the square $S_r$ in the plane $\R^2$, $A\cap \partial S_r \subset D$ and so does not belong to the boundary of $A$ in $D$.

The Cheeger constant is related to the $\L^1$ Poincar\'e inequality, and it is well known that 
\begin{equation}\label{eqcheegineq}
C_P \, \leq \, 4 \, C_C^2 \, ,
\end{equation}
 while $C_P$ can be finite but $C_C$ infinite. Hence an upper bound for the Cheeger constant will provide us with an upper bound for the Poincar\'e constant while a lower bound can only be a hint.
\medskip

\subsection{Squared obstacle.} \quad For simplicity we shall first assume that $d=2$, and use the notation in subsection \ref{subsecsquare}. Consider for $a>0$, the subset $A=\{x^1\geq a+r \, , \, |x^2|\leq r\}$ with boundary $\partial A =\{x^1\geq a+r \, , \, |x^2|= r\}$.\\ Recall the basic inequalities, for $0<b<c\leq +\infty$,
\begin{equation}\label{eqgauss}
\frac{b^2}{1+2b^2} \, \left(\frac{e^{- \, b^2}}{b} \, - \, \frac{e^{- \, c^2}}{c}\right) \leq \int_b^c \, e^{- \, u^2} \, du \leq \frac{1}{2b} \, \left(e^{- \, b^2} \, - \, e^{- \, c^2}\right) \, .
\end{equation}
It follows, for $r\sqrt \lambda$ large enough (say larger than one)
\begin{eqnarray*}
\frac{\mu_{\lambda,r}(A)}{\mu_{\lambda,r}^S(\partial A)} &=&  \, \frac{\left(\int_{a+r}^{+\infty} e^{-\lambda z^2} \, dz\right) \, \left(\int_{-r}^{+r} e^{-\lambda u^2} \, du\right)}{2 \, e^{- \lambda r^2} \, \left(\int_{a+r}^{+\infty} e^{-\lambda z^2} \, dz\right)}\\
&\geq& \frac{1}{2 \, \sqrt \lambda} \, e^{ \lambda r^2} \, \left(1 \, - \, \frac{1}{r \sqrt \lambda} \, e^{-\lambda r^2}\right) \, ,
\end{eqnarray*}
so that
\begin{equation}\label{eqcheeger1}
C_C(\lambda,y,r) \geq \frac{1}{2 \, \sqrt \lambda} \, e^{\lambda r^2} \, \left(1 \, - \, \frac{1}{r \sqrt \lambda} \, e^{-\lambda r^2}\right) \, .
\end{equation}
Note that this lower bound is larger than the one obtained by combining Cheeger's inequality \eqref{eqcheegineq} and the lower bound for the Poincar\'e constant obtained in Theorem \ref{thmsquare2}, since this combination furnishes an explosion like $e^{\lambda r^2/2}$.
\medskip

We strongly suspect, though we did not find a rigorous proof, that this set is ``almost'' the isoperimetric set, in other words that, up to some universal constant, the previous lower bound is also an upper bound for the Cheeger constant. In particular, we believe that this upper bound (hence the upper bound for the Poincar\'e constant) does not depend on $a$. Of course, since we know that the isoperimetric constant of the gaussian measure behaves like $1/\sqrt \lambda$, isoperimetric sets for the restriction of the gaussian measure to $D$ have some (usual) boundary part included in the boundary of the obstacle and our guess reduces to the following statement: if $r$ is large enough, for any subset $B \subset D$ with given gaussian measure, the standard gaussian measure of the part of the usual boundary of $B$ that does not intersect $\partial D$ is greater or equal to $C \, e^{-r^2}$ times the measure of the boundary intersecting $\partial D$.
\medskip

Of course, what we have just done immediately extends to $d$ dimensions, defining $A$
as $A=\{x^1\geq a+r \, , \, |x^i|\leq r \textrm{ for all } 2\leq i \leq d\}$ and furnishes exactly the same bound as \eqref{eqcheeger1} replacing $2$ by $2(d-1)$, i.e.  in dimension $d$
\begin{equation}\label{eqcheeger2}
C_C(\lambda,y,r) \geq \frac{1}{2(d-1) \, \sqrt \lambda} \, e^{\lambda r^2} \, \left(1 \, - \, \frac{1}{r \sqrt \lambda} \, e^{-\lambda r^2}\right) \, .
\end{equation}
In order to get a lower bound for the Poincar\'e constant, inspired by what precedes, we shall proceed as follows: denote by $A(u)$ the set $$A(u) = 
\{x^1\geq a+r \, , \, |x^i|\leq u \textrm{ for all } 2\leq i \leq d\} \, ,$$ and for $r>1$, choose a Lipschitz function $f$ such that $\BBone_{A(r-1)}\leq f \leq \BBone_{A(r)}$, for instance $f(x)=(1 - d(x,A(r-1)))_+$. 

If $Z_\lambda$ denotes the (inverse normalizing) constant in front of the exponential density of the gaussian kernel (notice that $Z_\lambda$ goes to $0$ as $\lambda$ goes to infinity), it holds 
\begin{eqnarray*}
\Var_{\mu_{\lambda,r}}(f) &\geq& \mu_{\lambda,r}(A(r-1)) - \left(\mu_{\lambda,r}(A(r))\right)^2\\ &\geq& Z_\lambda \, \int_{a+r}^{+\infty} e^{-\lambda u^2} du \, \left(\left(\int_{-r+1}^{r-1} e^{-\lambda u^2} du\right)^{d-1} -  \, Z_\lambda \, \left(\int_{-r}^{r} e^{-\lambda u^2} du\right)^{2(d-1)} \, \int_{a+r}^{+\infty} e^{-\lambda u^2} du\right) \, , 
\end{eqnarray*}
so that, there exists some universal constant $c$ such that, as soon as $r\sqrt \lambda>c$, $$\Var_\mu(f) \geq \frac{Z_\lambda}{2} \, \int_{a+r}^{+\infty} e^{-\lambda u^2} du \, \left(\int_{-r+1}^{r-1} e^{-\lambda u^2} du\right)^{d-1} \, .$$ At the same time again if $r\sqrt \lambda >c$,  $$\int |\nabla f|^2 \, d\mu_{\lambda,r} \leq \int \, \left(\BBone_{A(r)} - \BBone_{A(r-1)}\right) \, d\mu_{\lambda,r} \leq Z_\lambda \, \left(\int_{a+r}^{+\infty} e^{-\lambda u^2} du\right) \, \, \frac{e^{- \lambda (r-1)^2}}{(r-1) \lambda} \, (d-1) \, \left(\int_{-r}^{r} e^{-\lambda u^2} du\right)^{d-2} \, .$$ It follows, using homogeneity again, that 
\begin{eqnarray}\label{eqcpisop}
C_P(\lambda,y,r) \, &\geq& \, \frac 12 \, \left(\frac{r\sqrt \lambda -1}{(d-1) \lambda}\right) \, e^{(r\sqrt \lambda-1)^2} \,  \frac{\left(\int_{-r\sqrt \lambda+1}^{r\sqrt \lambda-1} e^{-
 u^2} du\right)^{d-1}}{\left(\int_{-r\sqrt \lambda}^{r \sqrt \lambda} e^{- u^2} du\right)^{d-2}} \nonumber \\ &\geq&  \left(\frac{r\sqrt \lambda -1}{(d-1) \lambda}\right) \, e^{(r\sqrt \lambda-1)^2} \,  \frac{1}{4\sqrt {\pi}} \, \left( 1 -\frac{e^{-(r \sqrt \lambda -1)^2}}{ r\sqrt \lambda -1}\right)^{d-2} \, .
\end{eqnarray}
Notice that this lower bound is smaller (hence worse) than the one we obtained in Theorem \ref{thmsquare3}, and also contain an extra  dimension dependent term (the last one). But of course it is much easier to get.\\ Since $1$ is arbitrary, we may replace $r \sqrt \lambda -1$ by $r \sqrt \lambda - \varepsilon$ for any $0\leq \varepsilon \leq 1$, the price to pay being some extra factor $\varepsilon^2$ in front of the lower bound for the Poincar\'e constant.

\subsection{Back to spherical obstacles. Another lower bound.}\label{lowerrondisop} \quad 
It is tempting to develop the same approach in the case of spherical obstacles. First assume $\lambda=1$.\\ Introduce for $0\leq u \leq r$, $$A(u) = 
\{x^1\geq a \, , \, |\bar x|\leq u \} \, \cap \, D \, .$$ As before we consider, for $\varepsilon<u$, a function $\BBone_{A(u-\varepsilon)}\leq f \leq \BBone_{A(u)}$ which is $1/\varepsilon$-Lipschitz. Then $$\Var_{\mu_{\lambda,r}}(f) \geq \mu_{\lambda,r}(A(u-\varepsilon)) - \left(\mu_{\lambda,r}(A(u))\right)^2 \, $$ and  $$\int |\nabla f|^2 \, d\mu_{\lambda,r} \leq (1/\varepsilon^2) \, \left(\mu_{\lambda,r}(A(u)) - \mu_{\lambda,r}(A(u-\varepsilon))\right) \, ,$$ with $$\mu_{\lambda,r}(A(u))= Z_\lambda \, \sigma(S^{d-2}) \, \int_{0}^u \, \left(\int_{a+\sqrt{r^2-s^2}}^{+\infty}  \, e^{-t^2} \, dt\right) \, s^{d-2} \, e^{-s^2} \, ds \, ,$$ and $\sigma(S^{d-2})$ is the Lebesgue measure of the unit sphere. It follows
\begin{eqnarray*}
(Z_\lambda)^{-1} \, \int |\nabla f|^2 \, d\mu_{\lambda,r} &\leq& (\sigma(S^{d-2})/\varepsilon^2) \, \int_{u-\varepsilon}^u \, \left(\int_{a+\sqrt{r^2-s^2}}^{+\infty} \, e^{-t^2} \, dt\right) s^{d-2} \, e^{-s^2} \, ds \\ &\leq& \frac{\sigma(S^{d-2})}{2 \varepsilon^2} \, \int_{u-\varepsilon}^u \, \frac{s^{d-2}}{(a+\sqrt{r^2-s^2})} \, e^{-(a+\sqrt{r^2-s^2})^2} \, e^{-s^2} \, ds\\ &\leq& \frac{\sigma(S^{d-2}) \, u^{d-2} \, e^{-(a^2+r^2)}}{2 \varepsilon^2 \, (a+\sqrt{r^2-u^2})} \, \int_{u-\varepsilon}^u \, e^{-2a \, \sqrt{r^2-s^2}} \, ds  \, .
\end{eqnarray*}
To get a precise upper bound for the final integral, we perform the change of variable $v=\sqrt{r^2-s^2}$ so that 
\begin{eqnarray*}
\int_{u-\varepsilon}^u \, e^{-2a \, \sqrt{r^2-s^2}} \, ds &=& \int^{\sqrt{r^2-(u-\varepsilon)^2}}_{\sqrt{r^2-u^2}} \, \frac{v}{\sqrt{r^2-v^2}} \, e^{-2av} \, dv \\ &\leq& \frac{\sqrt{r^2-(u-\varepsilon)^2}}{2a(u-\varepsilon)} \, \left(e^{-2a\sqrt{r^2-u^2}}-e^{-2a\sqrt{r^2-(u-\varepsilon)^2}}\right)\\ &\leq&  \frac{\sqrt{r^2-(u-\varepsilon)^2}}{2a(u-\varepsilon)} \, e^{-2a\sqrt{r^2-(u-\varepsilon)^2}} \, \left(e^{\frac{2a\varepsilon(2u-\varepsilon)}{\sqrt{r^2-(u-\varepsilon)^2}+\sqrt{r^2-u^2}}}-1\right) \, .
\end{eqnarray*}

Again for $r\geq c$ for some large enough $c$, and $a+\sqrt{r^2-u^2} \geq 1$, for $u>2\varepsilon$,
\begin{eqnarray*}
\Var_{\mu_{\lambda,r}}(f) &\geq& \frac 12 \, \mu_{\lambda,r}(A(u-\varepsilon)) \\ &\geq& \frac 12 \, Z_\lambda \, \sigma(S^{d-2}) \, e^{-(a^2+r^2)} \, \int_0^{u-\varepsilon} \, \frac{a+\sqrt{r^2-s^2}}{1+2(a+\sqrt{r^2-s^2})^2} \, s^{d-2} \, e^{-2a \, \sqrt{r^2-s^2}} \, ds\\  &\geq& \frac 12 \, Z_\lambda \, \sigma(S^{d-2}) \, e^{-(a^2+r^2)} \, \int_{u-2\varepsilon}^{u-\varepsilon} \, \frac{a+\sqrt{r^2-s^2}}{1+2(a+\sqrt{r^2-s^2})^2} \, s^{d-2} \, e^{-2a \, \sqrt{r^2-s^2}} \, ds \\  &\geq& \frac 12 \, Z_\lambda \, \sigma(S^{d-2}) \, e^{-(a^2+r^2)} \, \frac{a+\sqrt{r^2-(u-\varepsilon)^2}}{1+2(a+\sqrt{r^2-(u-2\varepsilon)^2})^2} \, (u-2\varepsilon)^{d-2} \, \\ && \quad \quad \quad \quad \quad  \int_{u-2\varepsilon}^{u-\varepsilon} \, e^{-2a \, \sqrt{r^2-s^2}} \, ds \, .
\end{eqnarray*}
The last integral is bounded from below by
\begin{eqnarray*}
\int_{u-2\varepsilon}^{u-\varepsilon} \, e^{-2a \, \sqrt{r^2-s^2}} \, ds
&\geq& \frac{\sqrt{r^2-(u-\varepsilon)^2}}{2a(u-\varepsilon)} \, e^{-2a\sqrt{r^2-(u-\varepsilon)^2}} \, \left(1 - e^{\frac{-2a\varepsilon(2u-3\varepsilon)}{\sqrt{r^2-(u-\varepsilon)^2}+\sqrt{r^2-(u-2\varepsilon)^2}}}\right)
\end{eqnarray*}


We thus deduce 
\begin{equation}\label{eqspherelow}
C_P(1,y,r) \geq \varepsilon^2 \, \frac{(a+\sqrt{r^2-u^2})(a+\sqrt{r^2-(u-\varepsilon)^2})}{1+2(a+\sqrt{r^2-(u-2\varepsilon)^2})^2} \, \frac{(u-\varepsilon)^{d-2}}{u^{d-2}} \, H \, ,
\end{equation}
 with $$H=\frac{1 - e^{\frac{-2a\varepsilon(2u-3\varepsilon)}{\sqrt{r^2-(u-\varepsilon)^2}+\sqrt{r^2-(u-2\varepsilon)^2}}}}{e^{\frac{2a\varepsilon(2u-\varepsilon)}{\sqrt{r^2-(u-\varepsilon)^2}+\sqrt{r^2-u^2}}}-1} \, .$$ Recall that for small $r$ (smaller than $c\sqrt{d-1}$ for some small enough $c$) we already know that $C_P(1,y,r)\geq c_d$, and presumably $c_d$ can be chosen independently of $d$. If $a-r>1$ we also know that $C_P(1,y,r)\geq c_d$.

The bound \eqref{eqspherelow} is not interesting if $a\gg r$, since in this case $H$ is very small, unless $\varepsilon$ is small enough (of order at most $r/a$), so that the lower bound we obtain goes to $0$ with $r/a$. Hence we shall only look at the case where $a/r \leq C$. Since $2\varepsilon <u<r$, for $H$ to be bounded from below by some universal constant, we see that $au\varepsilon \leq cr$ for some small enough universal constant $c$, so that we have to choose $u$ and $\varepsilon$ of order $\sqrt{r/a}$. It is then not difficult to see that, combined with all what we have done before, in particular in subsection \ref{subseclower}, this will furnish the following type of lower bound
\begin{proposition}\label{proplowerspherer}
There exists a constant $C_d$ such that for all $y$ and $r$, $$C_P(\lambda,y,r) \geq  \frac{C_d}{\lambda} \, \left(1+ \frac{r}{|y|\vee 1}\right) \, .$$
\end{proposition}
Even for very large $r's$, the previous method furnishes a dimension dependent bound. Proposition \ref{proplowerspherer} is interesting when the obstacle contains the origin, in which case we have a linear dependence in $r/|y|$. Of course, when $y=0$ we know that the lower bound growths as $r^2$.
\medskip

\subsection{How to kill the Poincar\'e constant with a far non convex obstacle.}\label{seckill}

In the previous section we have seen that when $\lambda$ goes to infinity an appropriately oriented  squared obstacle with any ``center'' furnishes a Poincar\'e constant that goes exponentially fast to infinity.

In this section, still in dimension $2$ for simplicity, we shall look at $\lambda=1$ with a non-convex bounded obstacle, namely we consider $$D^c=\{0\leq y-x^1\leq \alpha \, ; \, |x^2|\leq \alpha\} \, \cup \, \{0\leq x^1-y\leq \alpha \, ; \, \frac \alpha 2 \leq|x^2|\leq \alpha\} \, .$$ 

\begin{center}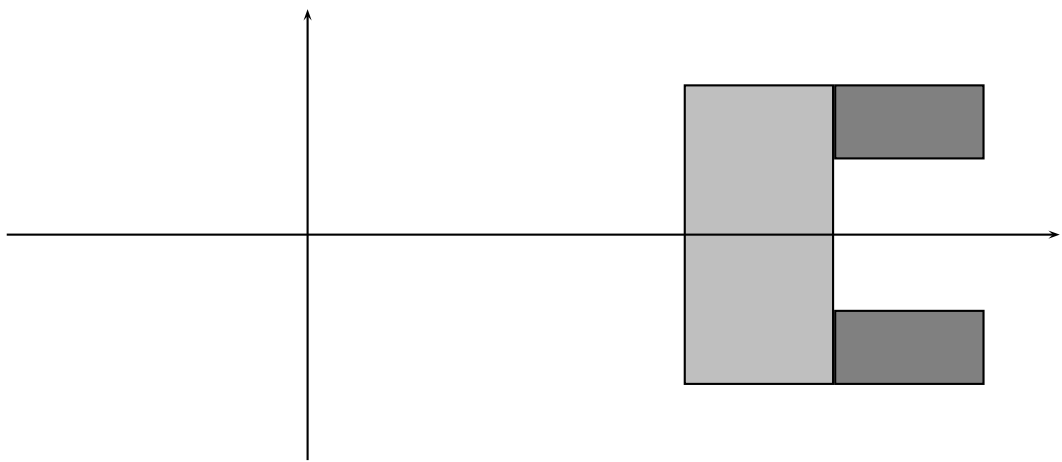
\begin{figure}[!h]
\begin{pspicture}(-7.5,-7.7)(7.5,7.7)
\psframe[fillstyle=solid,fillcolor=lightgray](2,-2)(4,2)
\psframe[fillstyle=solid,fillcolor=gray](4,1)(6,2)
\psframe[fillstyle=solid,fillcolor=gray](4,-1)(6,-2)
\psline{->}(-7,0)(7,0)
\psline{->}(-3,-3)(-3,3)
\end{pspicture}
\caption{A non convex trap}
\end{figure}
\end{center}

As in the end of the previous section we shall introduce some $2/\alpha$-Lipschitz function $f$ such that $\BBone_A \leq f \leq \BBone_B$ with $A=\{0\leq x^1-y\leq \frac \alpha 2 \, ; \, \frac \alpha 2 \geq|x^2|\}$ and $B=\{0\leq x^1-y\leq \alpha \, ; \, \frac \alpha 2 \geq|x^2|\}$. Hence $$\Var_{\mu}(f) \geq \mu(A) - \left(\mu(B)\right)^2 \quad \textrm{ and } \quad \int |\nabla f|^2 \, d\mu \leq \frac{4}{\alpha^2} \, \left(\mu(B) - \mu(A)\right) \, .$$ 
In addition $$\mu(A) = Z_1 \, \left(\int_y^{y+\frac\alpha 2} \, e^{-u^2} \, du\right) \, \left(\int_{-\frac \alpha 2}^{\frac\alpha 2} \, e^{-v^2} \, dv\right)  \, , \, \mu(B) = Z_1 \, \left(\int_{y}^{y+\alpha} \, e^{-u^2} \, du\right) \, \left(\int_{-\frac \alpha 2}^{\frac\alpha 2} \, e^{-v^2} \, dv\right) $$ so that 

\begin{eqnarray}\label{eqnonconv1}
\frac{\mu(A)}{\mu(B)} &\geq&  \frac{\int_{y}^{y+\frac \alpha 2} \, e^{-u^2} \, du}{\int_{y}^{y+\alpha} \, e^{-u^2} \, du} \geq \frac{\frac{y^2}{1+2y^2} \, \left(\frac{e^{-y^2}}{y} - \frac{e^{-(y+\frac \alpha2)^2}}{y+\frac \alpha2}\right)}{\frac{1}{2y} \, \left(e^{-y^2} - e^{- (\alpha+y)^2}\right)} \nonumber \\ &\geq& \frac{2y^2}{1+2y^2}\, \frac{1 - e^{-\alpha(y+\frac \alpha4)}}{1-e^{-\alpha(2y+\alpha)}}  \, ,
\end{eqnarray}
and
\begin{eqnarray}\label{eqnonconv}
\frac{\mu(A)}{\mu(B)-\mu(A)} &\geq&  \frac{\int_{y}^{y+\frac \alpha 2} \, e^{-u^2} \, du}{\int_{y+\frac \alpha2}^{y+\alpha} \, e^{-u^2} \, du} \geq \frac{\frac{y^2}{1+2y^2} \, \left(\frac{e^{-y^2}}{y} - \frac{e^{-(y+\frac \alpha2)^2}}{y+\frac \alpha2}\right)}{\frac{1}{2y} \, \left(e^{-(y+\frac \alpha2)^2} - e^{- (\alpha+y)^2}\right)} \nonumber \\ &\geq& \frac{2y^2}{1+2y^2}\, \frac{1 - e^{-\alpha(y+\frac \alpha4)}}{e^{-\alpha(y+\frac \alpha4)}-e^{-\alpha(2y+\alpha)}} \nonumber \\ &\geq& \frac{ 2 y^2 e^{\alpha(y+\frac \alpha4)}}{1+2y^2} \, \frac{1 - e^{-\alpha (y+\frac\alpha 4)}}{1 - e^{-\alpha(y+\frac{3\alpha}{4})}} \, .
\end{eqnarray}

$\mu(B)$ goes to $0$ as $y \to +\infty$ while there exists some constant $c$ such that $\mu(A) \geq c \, \mu(B)$, provided $\alpha$ is fixed and  $y$ large enough (depending on $\alpha$), in particular as soon as $\alpha y \to +\infty$. As previously we thus have for $\alpha y$ large enough, $\Var_{\mu}(f) \geq  \frac 12 \, \mu(A)$. Gathering all previous results, we thus get 
$C_P(\mu) \geq  
\frac 18 \, \frac{\mu(A)}{\mu(B)-\mu(A)}$ so that $C_P(\mu)$ explodes (at least) like $e^{\alpha y}$ if $\alpha y \to +\infty$. Hence, even a small non convex obstacle going to infinity, makes the Poincar\'e constant explode.
\medskip


\section{General spherical obstacles in dimension $2$.}\label{secunoby4}

What is the difference between a squared obstacle and a spherical obstacle ? We have seen that for a squared obstacle, the reflected O-U process starting from particular points in the shadow of  the obstacle, will spend a long time to go round the obstacle, i.e. is sticked near the boundary for a very long time. In the case of a spherical obstacle, one can think that the process will slide on the boundary much quickly. 

This behavior is connected with the spectral properties of the natural process on the boundary associated to the trace of the gaussian measure on this boundary.  What we shall do now is provide a new Lyapunov function near the obstacle. This Lyapunov function will be built by trying to understand how fast the process goes around the obstacle. 
\medskip

\subsection{The rate of rotation.}\label{subsecrotate} \quad 
To this end we introduce a new stochastic process $Y_t$ which is just the reflected Ornstein-Uhlenbeck process in the shell $S=\{r \leq |x-y| \leq r+q\}$ for some positive $q$, i.e
\begin{equation}\label{eqOUR2}
\left\{
\begin{array}{l}
dY_t =  dW_t \, - \, \lambda \, Y_t \, dt \, +  \, (Y_t - y) \, dL_t  \, ,\\
L_t = \int_0^t \, \left(\BBone_{|Y_s - y|=r} - \BBone_{|Y_s - y|=r+q}\right) \, dL_s .
\end{array}
\right.
\end{equation}
Next as usual, we assume that $y=(a,0)$ and write the generic point of the euclidean space as $x=(x^1,\bar x)$. Again $n$ denotes the normal vector field $(x^1-a, \bar x)$ (pointing either inward or outward), so that, for any nice function $g$, Ito formula yields $$g(Y_t)=g(Y_0)+\int_0^t \nabla g(Y_s).dW_s + \int_0^t Lg(Y_s) ds + r \, \int_0^t \frac{\partial g}{\partial n}(Y_s) \, dL_s \, .$$

Finally we shall look at the process 
\begin{equation}\label{eqangle}
Z_t = \arccos \left( \frac{Y_t^1 - a}{\sqrt{|\bar Y_t|^2 + (Y_t^1 - a)^2}}\right) \, = \, \varphi(Y_t) \, .
\end{equation}
We can calculate $$\nabla \varphi(x)= \left(\frac{-|\bar x|}{(x^1-a)^2+|\bar x|^2} \, , \, 
\frac{(x^1-a) \, \bar x}{|\bar x|\left((x^1-a)^2+|\bar x|^2\right)}\right) \quad \textrm {so that } \quad \frac{\partial \varphi}{\partial n} \, (x) = 0 \, .$$
Consider $M=\{-r-q\leq x^1-a\leq -r \, , \, \bar x=0\}$. If $Y_0 \notin M$, i.e. $Z_0 \neq \pi$, we may apply Ito-Tanaka formula up to time $T_M$ (the first time $Y_.$ hits $M$) yielding for $t<T_M$,
\begin{eqnarray}\label{eqZ}
Z_t^2 &=& Z_0^2 + \int_0^t \, 2Z_s \, \langle \nabla \varphi(Y_s),dW_s\rangle + \int_0^t \, |\nabla \varphi(Y_s)|^2 \, ds \\ & & \, + \int_0^t \, \frac{2 \, Z_s (\lambda \, a  \, \, |\bar Y_s| + (d-2) \, (Y^1_s-a))}{|\bar Y_s|^2+(Y_s^1-a)^2} \, ds \nonumber \\ &=& Z_0^2 + \int_0^t \, \frac{2  Z_s}{\left(|\bar Y_s|^2+(Y_s^1-a)^2\right)^{1/2}} \, dB_s + \int_0^t \, \frac{1 + 2 \, Z_s (\lambda \, a  \, \, |\bar Y_s| + (d-2) \, (Y^1_s-a))}{|\bar Y_s|^2+(Y_s^1-a)^2} \, ds \nonumber
\end{eqnarray}
where $B_.$ is a new standard Brownian motion. We have considered $Z^2$ instead of $Z$  to kill the local time at $0$ of $Z_.$ (since $t<T_M$ the local time of $Z_.$ at $\pi$ does not appear too). 

We want to estimate $T_M$ by comparing $Z_t$ with a simpler diffusion process for which estimates are easy to obtain (since they are known). The problem is that the sign of the drift term may change, except if the possibly negative part of it vanishes. This is the case if $\bf{d=2}$, so that from now on we assume that $d=2$.

Set $$A(t)=\int_0^t \, \frac{1}{\left(|\bar Y_s|^2+(Y_s^1-a)^2\right)} \, ds \, ,$$ and $A^{-1}(t)$ the inverse of $A(.)$. \\ Notice that $(t/(r+q)^2) \leq A(t) \leq (t/r^2)$ so that $r^2 t \leq A^{-1}(t) \leq (r+q)^2 t$.

Define $\tilde Y_t = Y_{A^{-1}(t)}= (\tilde Y_t^1,V_t)$ and $U_t = Z^2_{A^{-1}(t)}$. Then for $t<A(T_M)$, $U_.$ satisfies
\begin{equation}\label{eqtimechange}
U_t = Z_0^2 + \int_0^t \, 2 \, \sqrt{U_s} \, d\tilde B_s + \int_0^t \, \left(1 + 2 \lambda \, a  \, \sqrt{U_s} \, |V_s|\right) ds \, ,
\end{equation}
for some new Brownian motion $\tilde B_.$. We can thus compare, $U_.$ with the square of a Bessel process of dimension 1, i.e. the square of the reflected Brownian motion. Comparison of one dimensional diffusion processes yields that, almost surely $$U_t \geq |Z_0+\tilde B_t|^2 \, .$$ Hence for $t<T_M$, 
\begin{equation}\label{eqcompareBes}
Z_t \geq |Z_0 + \tilde B_{A(t)}| \, .
\end{equation}
It follows that $T_M$ is smaller than the first time the process $|Z_0+\tilde B_{./(r+q)^2}|$ hits $\pi$, which is itself smaller than the first time $|\tilde B_{./(r+q)^2}|$ hits $2 \pi$. This latter process has the same law as $1/(r+q)$ times a standard Brownian motion. We have recalled in subsection \ref{subsecsquare} that the exit time of $[-b,b]$ for a standard Brownian motion has exponential moments up to $\theta < (\pi^2/8 b^2)$. We deduce for all this that for all $x\notin M$,
\begin{equation}\label{eqexittour}
\mathbb E_x \left(e^{\theta \, T_M}\right) < +\infty \quad \textrm{ provided } \quad \theta < 1/(32(r+q)^2) \, .
\end{equation}

It is thus tempting to define $W(x)=\mathbb E_x \left(e^{\theta \, T_M}\right)$, which satisfies $LW=-\theta W$ in $S-M$. But $W$ is not regular enough (the second derivative contains some Dirac mass on $M$), hence we cannot apply Lemma \ref{lemIPP}, or, if we apply it, a new ``boundary term'' on $M$ will appear.\\
However, since $W$ is a Lyapunov function out of $M$, we may try to build another Lyapunov function near $M$ (actually $|x|^2$ is a good choice), and combine them in order to define a good Lyapunov function in $S$. This method is used for instance in \cite{BHW} and \cite{AKM}.  

Here we shall use another method.
\medskip

\subsection{The Poincar\'e inequality in the shell $S$.}\label{poincshell} \quad Using what precedes we shall prove the following first result 
\begin{proposition}\label{proppoincshell}
Let $s$ be such that $(r+s)^2 + s^2< (r+q)^2$, and assume that $a>1+r+s$. Then, the (non normalized) restriction of $\mu_{1,r}$ to the shell $S=\{r \leq |x-y| \leq r+q\}$ satisfies a Poincar\'e inequality $$\int_S \, f^2 \, d\mu_{1,r} \leq C_P(S) \, \int_S \, |\nabla f|^2 \, d\mu_{1,r} + \frac{1}{\mu_{1,r}(S)} \, \left(\int_S \, f \, d\mu_{1,r}\right)^2 \, $$ where $$C_P(S) \leq 64 \, (r+q)^2 + \left(1 + \frac{64 (r+q)^2}{s^2}\right) \, \left(\frac 52+ \frac{1}{s^2}\right) .$$
\end{proposition}
\begin{proof}
We shall use the results in the previous subsection. Introduce the subset $$K = \{x^1 -a < 0 \, , \, |\bar x| \leq \eta < r\} \cap S \, .$$ Since $M \subset K$ we know from \eqref{eqexittour} that the hitting time $T_K$ has an exponential moment of order $\theta$ for $\theta < 1/(32(r+q)^2)$. Define $W(x)=\mathbb E_x\left(e^{\theta T_K}\right)$ for $x \in S$. Then $W$ belongs to the domain of the generator of $Y_.$ (in particular the normal derivative on the shell's boundary vanishes) and satisfies $LW =-\theta W$ in $S-K$. 

Consider now $$K'= \left\{x^1 -a < 0 \, , \, |\bar x| \leq \eta+s  < \sqrt{(r+q)^2-r^2}\right\} \cap S \, .$$ Then as before, using \cite{CGZ} formula (2.14) (in the present framework of our reflected Ornstein-Uhlenbeck process $Y_.$), we have 
\begin{equation}\label{eqshell1}
C_P(S) \leq \frac 2\theta + \left(\frac{2}{\theta \, s^2} + 1\right)C_P(K') \, .
\end{equation}
It remains to get some bound for $C_P(K')$. 
\smallskip

Again we divide $K'$ in two overlapping parts: $$K'_1 = \left\{ - \sqrt{(r+q)^2 -(\eta+s)^2}<-r -s < x^1 -a < 0 \, , \, |\bar x| \leq \eta+s  < \sqrt{(r+q)^2-r^2}\right\} \cap S$$ and $$K'_2 = \left\{ x^1 -a < -r \, , \, |\bar x| \leq \eta+s  < \sqrt{(r+q)^2-r^2}\right\} \cap S \, .$$ Note that $K'_2$ is convex. Hence the restriction of the gaussian measure to $K'_2$ satisfies a Poincar\'e inequality with constant $1/2$.

As before it is then sufficient to build some Lyapunov function in $K'_1$. This time we choose $W(x)=(x^1)^2$. Note that, on one hand, the normal derivative of $W$ on $|\bar x| =  \eta+s$ is equal to $0$, while on the other hand, the (non normalized) inward normal derivative of $W$ on $|x-y|=r$ is equal to $2(x^1-a)x^1$. The latter is thus negative provided $x^1>0$, hence in particular if $a>r+s$.\\ In addition, $$LW(x)=1 - 2(x^1)^2\leq -(x^1)^2 \quad \textrm{ in } K'_1$$ as soon as $a>r+1+s$.
\smallskip

Thus, as before we obtain $$C_P(K') \leq 2 + \left(\frac{1}{s^2} + \frac 12\right)  \, .$$ The proof is completed, since we may take $\eta$ as small as we want.
\end{proof}

\subsection{A new estimate for an obstacle which is not too close to the origin.\\}\label{subsecnewestimate} 


 We may use Proposition \ref{proppoincshell} to build a new Lyapunov function near the obstacle. \\In the situation of the proposition consider $T_{S/2}$ the hitting time of the ``half'' shell $S'=\{r+(q/2) \leq |x-y| \leq r+q\}$.
Then according to proposition \ref{prophit} we may define $W(x)=\mathbb E_x\left(e^{\theta T_{S/2}}\right)$ which satisfies $LW=-\theta W$ for $x\in S-S'$ and $\partial W/\partial n =0$ on $|x-y|=r$, provided
 $$\theta < \frac{1}{8 C_P(S)} \, \, \frac{\mu_{1,r}(S')}{\mu_{1,r}(S)} \, .$$ Now we can apply lemma \ref{lemlypoids} with $2h=q/2$, provided $\theta h^4 >1$.

It remains to choose all parameters. All conditions are satisfied for instance if $$\frac{q^4}{4^4} \, \frac{1}{16 C_P(S)} \, \, \frac{\mu_{1,r}(S')}{\mu_{1,r}(S)} \, >1 \, .$$ We have clearly to choose something like $s=1$, $q>2$ so that the condition in Proposition \ref{proppoincshell} is satisfied and $C_P(S) \leq  4 + 320 (r+q)^2$. Choosing $q=c(1+\sqrt r)$ for a large enough (universal) constant $c$, we see that the previous condition reduces to $$C \, (1+r^2) > \frac{\mu_{1,r}(S)}{\mu_{1,r}(S')} \, ,$$ for a large enough universal $C$ (the larger $c$, the larger $C$). 

It is not too difficult to be convinced that the ratio of the two measures is uniformly (in $r$ and $y$) bounded above, provided $a-r-q >1$ ($1$ can be replaced by any positive constant), i.e. provided the origin is far enough from $B(y,r+q)$. Indeed the measure restricted to $S$ is mainly concentrated near the point $(a-r-q,0)$ which belongs to both $S$ and $S-S'$.

This amounts to $a>1+r+c(1+\sqrt r)$. Putting all this together we have obtained the following result:

\begin{theorem}\label{thm2d}
Assume $d=2$. One can find universal constants $c$ and $C$ such that, provided $|y|>1+r+c(1+\sqrt r)$, $$C_P(1,y,r) \, \leq \, C \, (1+r^2) \, .$$
\end{theorem} 

\begin{remark}\label{remd2}
The extension of this result to any dimension is presumably possible by founding a well adapted $\varphi$, but presumably again $c$ and $C$ will depend on $d$. But if we compare this conjectured result with Theorem \ref{thmmanu}, we see that when $|y|>5r$, the latter furnishes a bound like $c (1+r^2+r^4)$ with some $c$ independent of the dimension. Of course the polynomial dependence in $r$ is worse for large $r$'s, but the approach seems much more simple.
\end{remark}

\bigskip

\bibliographystyle{alpha} 
\bibliography{punctured}

\end{document}